\documentclass[a4paper, 12pt]{article}

\usepackage[cp1250]{inputenc}
\usepackage{amsmath}
\usepackage{amsthm}
\usepackage{amssymb}
\usepackage{mathrsfs}
\usepackage{graphicx}
\usepackage{proof}
\usepackage{wrapfig}
\usepackage{authblk}
\usepackage{caption}
\usepackage{subcaption}

\usepackage[left=1.5cm, right=1.5cm, top=1.5cm, bottom=1.5cm, headsep=1.2cm]{geometry}
\title{Intuitionistic modal logic based on neighborhood semantics without superset axiom}
\author{Tomasz Witczak \\ tm.witczak@gmail.com}

\affil{Institute of Mathematics, University of Silesia, Katowice}
\date{January 2018}
\usepackage{amsthm}

\theoremstyle{Theorem}
\newtheorem{tw}{Theorem}[section]

\theoremstyle{Lemma}
\newtheorem{lem}{Lemma}[section]

\theoremstyle{Definition}
\newtheorem{df}{Definition}[section]

\theoremstyle{Remark}

\usepackage{graphicx}

\begin{document}

\maketitle

\begin{abstract}
In this paper we investigate certain systems of propositional intuitionistic modal logic defined semantically in terms of neighborhood structures. We discuss various restrictions imposed on those frames but our constant approach is to discard \textit{superset} axiom. Such assumption allows us to think about specific modalities $\Delta, \nabla$ and new functor $\rightsquigarrow$ depending on the notion of maximal neighborhood. We show how it is possible to treat our models as bi-relational ones. We prove soundness and completeness of proposed axiomatization by means of canonical model. Moreover, we show that finite model property holds. Then we describe properties of \textit{bounded morphism}, \textit{behavioral equivalence}, \textit{bisimulation} and \textit{n-bisimulation}. Finally, we discuss further researches (like some interesting classical cases and particular topological issues).
\end{abstract}

\section*{Introduction}

Our general purpose is to establish sound and complete neighborhood semantics for propositional intuitionistic modal logic. Moreover, we want to investigate and describe some fundamental properties of structures achieved. Basically, we follow the standard path, speaking about the universe $W$ of possible worlds. We assume that each world $w$ has its minimal and maximal neighborhood. The former simulates intuitionistic reachability of other worlds; the latter refers to the modal visibility (and can be properly contained in $W$, i.e. there is no \textit{superset} axiom).

Such approach allows us also to think about specific modality $\Delta$, behaving in maximal neighborhood partially like necessity. Also we discuss $\nabla$ - possibility operator which is not dual to $\Delta$. We show which well-known modal axioms hold for those operators and which are rejected.

 Later, we introduce new functor of implication ($\rightsquigarrow$) depending on the notion of maximal neighborhood. There can also be established certain correspondence between maximal neighborhoods and topological spaces. Even more interesting is bi-relational point of view. In this setting, we prove finite model property. As for completeness results, we propose canonical models based on prime theories. It should be stated out that \textit{fmp} and completeness are obtained only for the mono-modal case (in the language containing single modal operator, that of necessity).

After those intuitionistic considerations, we present few classical modal systems with two operators: $\Box$ and $\Delta$. While $\Delta$ still refers to the satisfaction of formula in the whole maximal neighborhood, $\Box$ describes necessity in minimal neighborhood. We obtain completeness results for those systems and the we show how to establish translation between one of them and our initial intuitionistic modal logic.

\section{Alphabet and language}

Our basic system is named \textbf{IML1}. It has rather standard syntax (i.e. alphabet and language). We use the following notations:

\begin{enumerate}
\item $PV$ is a fixed denumerable set of propositional variables $p, q, r, s, ...$
\item Logical connectives and operators are $\land$, $\lor$, $\rightarrow$, $\bot$, $\Delta$.
\item The only derived connective is $\lnot$ (which means that $\lnot \varphi$ is a shortcut for $\varphi \rightarrow \bot$).
\end{enumerate}

Formulas are generated recursively in a typical way: if $\varphi$, $\psi$ are \textit{wff's} then also $\varphi \lor \psi$, $\varphi \land \psi$, $\varphi \rightarrow \psi$ and $\Delta \varphi$. Semantical interpretation of propositional variables and all the connectives introduced above will be presented in the next section. Attention: $\Leftarrow, \Rightarrow$ and $\Leftrightarrow$ are used only on the level of meta-language (which is obviously classical).

\section{Neighborhood semantics}

\subsection{Foundations of neighborhood semantics}
Neighborhood semantics for (classical) modal logics was introduced by Montague (1968), Scott (1970) and Segerberg (1968). Nowadays it is still investigated (see interesting surveys by Pacuit [6] and Shehtman [8]). The main idea is that we have model $\textbf{M} = \langle W, \mathcal{N}, V \rangle$ where $\mathcal{N}: W \rightarrow P(P(W))$ is so-called \textit{neighborhood function} and $V: PV \rightarrow P(W)$ is valuation. Each world has its (possibly empty) family of neighborhoods (which are just subsets of universe $W$). Standard interpretation of modal operators is shown below:

$w \Vdash \Box \varphi$ $\Leftrightarrow$ there is $U \in \mathcal{N}_{w}$ such that for every $v \in U$ we have $v \Vdash \varphi$

$w \Vdash \Diamond \varphi$ $\Leftrightarrow$ $-V(\varphi) \notin \mathcal{N}_{w}$

The main advantage of neighborhood semantics is that it guarantees completeness results even for very weak modal logics (in particular, for non-regular and non-monotonic systems). The only condition is that \textbf{E} $\subseteq L$. \textbf{E} is the weakest classical modal logic and it contains only instances of classical propositional axioms, duality axiom ($\Box \varphi \leftrightarrow \lnot \Diamond \lnot \varphi$), rule \textit{modus ponens} and the rule of extensionality \textbf{RE} ($\varphi \leftrightarrow \psi$ $\Rightarrow$ $\Box \varphi \leftrightarrow \Box \psi$).

As for the neighborhood semantics for intuitionistic logic, it has been presented by Moniri and Maleki in [4]. They have also considered simple subintuitionistic case of \textbf{BPL} (Basic Propositional Logic).

\subsection{The definition of structure}

\begin{df}
Our basic structure is a neighborhood frame (\textbf{nIML1}-frame) defined as an ordered pair $\langle W, \mathcal{N} \rangle$ where:

\begin{enumerate}
\item $W$ is a non-empty set (of worlds, states or points)

\item $\mathcal{N}$ is a function from $W$ into $P(P(W))$ such that:

\begin{enumerate}

\item $w \in \bigcap \mathcal{N}_{w}$

\item $\bigcap \mathcal{N}_{w} \in \mathcal{N}_{w}$

\item $u \in \bigcap \mathcal{N}_{w}$ $\Rightarrow$ $\bigcap \mathcal{N}_{u} \subseteq \bigcap \mathcal{N}_{w}$ ($\rightarrow$-\textit{condition})

\item $X \subseteq \bigcup \mathcal{N}_{w}$ and $\bigcap \mathcal{N}_{w} \subseteq X$ $\Rightarrow$ $X \in \mathcal{N}_{w}$ (\textit{relativized superset axiom})

\item $u \in \bigcap \mathcal{N}_{w} \Rightarrow \bigcup \mathcal{N}_{u} \subseteq \bigcup \mathcal{N}_{w}$ \textit{($\Delta$-condition)}.

\end{enumerate}

\end{enumerate}

\end{df}

\subsection{Valuation and model}

\begin{df}
Neighborhood \textbf{nIML1}-model is a triple $\langle W, \mathcal{N}, V \rangle$, where $\langle W, \mathcal{N} \rangle$ is an \textbf{nIML1}-frame and $V$ is a function from $PV$ into $P(W)$ satisfying the following condition: if $w \in V(q)$ then $\bigcap \mathcal{N}_{w} \subseteq V(q)$.
\end{df}

\begin{df}
Forcing of formulas in a world $w \in W$ is defined inductively:
\end{df}
\begin{enumerate}

\item $w \Vdash q$ $\Leftrightarrow$ $w \in V(q)$ for any $q \in PV$

\item $w \Vdash \varphi \rightarrow \psi$ $\Leftrightarrow$ $\bigcap \mathcal{N}_{w} \subseteq \{v \in W; v \nVdash \varphi$ or $v \Vdash \psi\}$.

\item $w \Vdash \lnot \varphi$ $\Leftrightarrow$ $w \Vdash \varphi \rightarrow \bot$ $\Leftrightarrow$ $\bigcap \mathcal{N}_{w} \subseteq \{v \in W; v \nVdash \varphi\}$

\item $w \Vdash \Delta \varphi$ $\Leftrightarrow$ $\bigcup \mathcal{N}_{w} \subseteq \{v \in W; v \Vdash \varphi\}$.

\item $w \nVdash \bot$

\end{enumerate}

The reader could say that the symbol $\Delta$ is redundant and that typical $\Box$ would be appropriate. In fact, in intuitionistic modal systems $\Box$ is often defined just like our modality (albeit in bi-relational or algebraic setting). In other words, necessity of $\varphi$ in a given world $w$ means that $\varphi$ is accepted in all worlds which are \textit{modally visible (reacheable)} from $w$. We understand such objection but later we will use $\Box$ symbol for different goals.

There is also one technical annotation: sometimes we will write $X \Vdash \varphi$ where $X$ would be a subset of $W$, in particular - minimal or maximal neighborhood (e.g. $\bigcap \mathcal{N}_{w} \Vdash \varphi$). Clearly it would mean that each element of $X$ forces $\varphi$.

As usually, we say that formula $\varphi$ is satisfied in a model $M = \langle W, \mathcal{N}, V \rangle$ when $w \Vdash \varphi$ for every $w \in W$. It is \textit{true} (tautology) when it is satisfied in each \textbf{nIML1}-model.

\subsection{Equivalent definitions of $\rightarrow$- and $\Delta$-conditions}

\begin{lem}
$\rightarrow$-condition is equivalent to the following property:

$$\bigcap \mathcal{N}_{w} \subseteq X \Rightarrow \bigcap \mathcal{N}_{w} \subseteq \{v \in W; \bigcap \mathcal{N}_{v} \subseteq X\}$$

\end{lem}

\begin{proof}

$\Rightarrow$

Take $w \in W$ and assume that $\bigcap \mathcal{N}_{w} \subseteq X \subseteq W$. Now consider $Y = \{v \in W; \bigcap \mathcal{N}_{v} \subseteq X\}$. Suppose that $\bigcap \mathcal{N}_{w} \nsubseteq Y$, so there is certain $u \in \bigcap \mathcal{N}_{w}$ such that $u \notin Y$. This gives us that $\bigcap \mathcal{N}_{u} \nsubseteq X$, so there is $r \in \bigcap \mathcal{N}_{u}$ such that $r \notin X$. But if $u \in \bigcap \mathcal{N}_{w}$ then $\bigcap \mathcal{N}_{u} \subseteq \mathcal{N}_{w} \subseteq X$. Then $r \in X$ - and this is our expected contradiction.

$\Leftarrow$

Assume that $u \in \bigcap \mathcal{N}_{w}$ for certain $w \in W$. Of course $\bigcap \mathcal{N}_{u} \subseteq \bigcap \mathcal{N}_{u}$ and then (by the assumption) $\bigcap \mathcal{N}_{u} \subseteq \{v \in W; \bigcap \mathcal{N}_{v} \subseteq \bigcap \mathcal{N}_{u}\}$. So if $u \in \bigcap \mathcal{N}_{w}$ then $\bigcap \mathcal{N}_{u} \subseteq \bigcap \mathcal{N}_{w}$.

\end{proof}

\begin{lem}
$\Delta$-condition is equivalent to the following property:

$$\bigcup \mathcal{N}_{w} \subseteq X \Rightarrow \bigcap \mathcal{N}_{w} \subseteq \{v \in W; \bigcup \mathcal{N}_{v} \subseteq X\}$$

\end{lem}

\begin{proof}

$\Rightarrow$

Suppose that there is $w \in W$ such that $\bigcup \mathcal{N}_{w} \subseteq X$. Put $Y = \{v \in W; \bigcup \mathcal{N}_{v} \subseteq X\}$ and assume that $\bigcap \mathcal{N}_{w} \nsubseteq Y$. Then we have $u \in \bigcap \mathcal{N}_{w}$ such that $u \notin Y$. Thus there is $r \in \bigcup \mathcal{N}_{u}$ such that $r \notin X$. But (by $\Delta$-condition) $\bigcup \mathcal{N}_{u} \subseteq \bigcup \mathcal{N}_{w} \subseteq X$ and this is contradiction.

$\Leftarrow$

Let $u \in \bigcap \mathcal{N}_{w}$. We can always say that $\bigcup \mathcal{N}_{w} \subseteq \bigcup \mathcal{N}_{w}$. Then $\bigcap \mathcal{N}_{w} \subseteq \{v \in W; \bigcup \mathcal{N}_{v} \subseteq \bigcup \mathcal{N}_{w}\}$. Then of course $\bigcup \mathcal{N}_{u} \subseteq \bigcup \mathcal{N}_{w}$.

\end{proof}

\subsection{Monotonicity of forcing in minimal neighborhood}

The following theorem is crucial for the intuitionistic aspect of our logic. It is analogous to the theorem about monotonicity of forcing in relational semantics. Note that we do not have such property for maximal neighborhoods.

\begin{tw}
If $w \Vdash \varphi$ then $\bigcap \mathcal{N}_{w} \subseteq V(\varphi)$.
\end{tw}

\begin{proof}
For $\land$ and $\lor$ this proof goes by induction over the complexity of formulas - but for $\rightarrow$ we must use $\rightarrow$-condition and for $\Delta$ we need $\Delta$-condition. Details are listed below.
\begin{enumerate}
\item $w \Vdash q$. Our expected conclusion is immediate (from the very definition of valuation).

\item $w \Vdash \varphi \rightarrow \psi$. It means that $\bigcap \mathcal{N}_{w} \subseteq \{z \in W; z \nVdash \varphi$ or $z \Vdash \psi \}$. On the other side we can write that $V(\varphi \rightarrow \psi) = \{u \in W; u \Vdash \varphi \rightarrow \psi\} = \{u \in W; \bigcap \mathcal{N}_{u} \subseteq \{z \in W; z \nVdash \varphi$ or $z \Vdash \psi\}\}$. But this clearly gives us (by $\rightarrow$-condition and Lemma 2.1) that $\bigcap \mathcal{N}_{w} \subseteq V(\varphi \rightarrow \psi)$.

\item $w \Vdash \Delta \varphi$. This means that $\bigcup \mathcal{N}_{w} \subseteq Y = \{v \in W; v \Vdash \varphi\}$. On the other side, $V(\Delta \varphi) = \{s \in W; s \Vdash \Delta \varphi\} = \{s \in W; \bigcup \mathcal{N}_{s} \subseteq Y\}$. Now assume that $\bigcap \mathcal{N}_{w} \nsubseteq V(\Delta \varphi)$. Then there exists $x \in \bigcap \mathcal{N}_{w}$ such that $x \notin V(\Delta \varphi)$. Now we have $y \in \bigcup \mathcal{N}_{x}$ such that $y \nVdash \varphi$. But $\bigcup \mathcal{N}_{x} \subseteq \bigcup \mathcal{N}_{w} \subseteq Y = V(\varphi)$. This is plain contradiction.

\end{enumerate}

\end{proof}

\subsection{Importance of $\Delta$-condition}

Without $\Delta$-condition Theorem 1 would be impossible. We can easily find a counter-example.

$W = \{w, v, z\}$, $\bigcap \mathcal{N}_{w} = \bigcup \mathcal{N}_{w} = \{w, v\}$, $\bigcap \mathcal{N}_{v} = \{v\}, \bigcup \mathcal{N}_{v} = \{v, z\}$, $\bigcap \mathcal{N}_{z} = \bigcup \mathcal{N}_{z} = \{z\}$, $V(\varphi) = \{w, v\}$

\begin{figure}[ht]
\centering
\includegraphics[height=6cm]{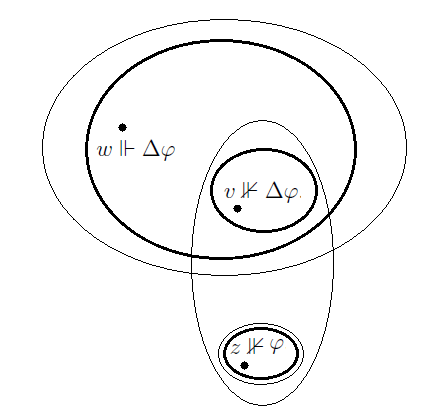}
\caption{Counter-example for Th. 2.1. when $\Delta$-condition is not introduced}
\label{fig:obrazek {pic1_delta}}
\end{figure}

Let us treat $v$ as an element of minimal $w$-neighborhood. Monotonicity theorem would imply that $v \Vdash \Delta \varphi$ (just like $w$). But $v \nVdash \Delta \varphi$ because of $z$ which belongs to maximal $v$-neighborhood and does not force $\varphi$. Note that $z \notin \bigcup \mathcal{N}_{w}$ so $\bigcup \mathcal{N}_{v} \nsubseteq \bigcup \mathcal{N}_{w}$. Thus, $\Delta$-condition is not satisfied.

\section{Intuitionistic and modal axioms}

One can easily check that purely intuitionistic axioms are satisfied in our system. In fact, someone familiar with intuitionistic calculus, Kripke models and upsets will quickly perceive this fact, having in his mind $\rightarrow$-condition and monotonicity theorem. Now a natural question arises: which well-known modal axioms and rules are satisfied in our logic? The answer is below.

\begin{enumerate}
\item \textbf{K}: $\Delta (\varphi \rightarrow \psi) \rightarrow (\Delta \varphi \rightarrow \Delta \psi)$.

Suppose that there is \textbf{nIML1}-model $\langle W, \mathcal{N}, V \rangle$ with an element $w \in W$ such that $w \nVdash$ \textbf{K}. This means that there exists $z \in \bigcap \mathcal{N}_{w}$ which forces $\Delta(\varphi \rightarrow \psi)$ but does not accept $\Delta \varphi \rightarrow \Delta \psi$. So for every $x \in \bigcup \mathcal{N}_{z}$, $x \Vdash \varphi \rightarrow \psi$ which means that for every $y \in \bigcap \mathcal{N}_{x}$ we have $y \nVdash \varphi$ or $y \Vdash \psi$.

On the other side there exists certain $v \in \bigcap \mathcal{N}_{z}$ accepting $\Delta \varphi$ but not $\Delta \psi$. Then $\bigcup \mathcal{N}_{v} \Vdash \varphi$ - but there is also $r \in \bigcup \mathcal{N}_{v}$ such that $r \nVdash \psi$.

But if $v$ is in minimal $z$-neighborhood, then $\bigcup \mathcal{N}_{v} \subseteq \bigcup \mathcal{N}_{z}$ (from $\Delta$-condition). This means that $r \in \bigcup \mathcal{N}_{z}$. Then $r \nVdash \varphi$ or $r \Vdash \psi$ - although we said that $r \Vdash \varphi$ and $r \nVdash \psi$. Clearly this is a contradiction.

\item \textbf{T}: $\Delta \varphi \rightarrow \varphi$.

Obviously this condition is always satisfied (because each element $w$ of our universe is contained in its minimal - and thus also maximal - neighborhood).

\item \textbf{MP} (\textit{modus ponens})

Suppose that this rule is not satisfied. Then we have $W, w \in W$ such that $w$ does not force certain $\psi$ although each point forces $\varphi \rightarrow \psi$ and $\varphi$ (for certain $\varphi$). In particular, $w \nVdash \varphi$ or $w \Vdash \psi$ - but $w \Vdash \varphi$, so $w \Vdash \psi$. Contradiction.

\item \textbf{RN} (\textit{necessity rule}): $\varphi$ $\Rightarrow$ $\Delta \varphi$

This condition is satisfied. Suppose the opposite, i.e. that there are $W, w \in W$ and $\varphi$ such that $w \nVdash \Delta \varphi$ although each point forces $\varphi$. Now it means that there is $v \in \bigcup \mathcal{N}_{w}$ such that $v \nVdash \varphi$. Contradiction because $\varphi$ is forced everywhere in $W$.

\end{enumerate}

At this moment we can stop our considerations because only this package of axioms will be later used in completeness proof. However, we can point out that also the following rules are valid:

\begin{enumerate}

\item \textbf{D}: $\Delta \varphi \rightarrow \lnot \Delta \lnot \varphi$

\item \textbf{RM} (\textit{monotonicity rule}): $\varphi \rightarrow \psi$ $\Rightarrow$ $\Delta \varphi \rightarrow \Delta \psi$

\end{enumerate}

Below we see some typical modal formulas which are not true in our environment.

\begin{enumerate}
\item \textbf{4}: $\Delta \varphi \rightarrow \Delta \Delta \varphi$

Note that axiom \textbf{4} becomes true if we put the following restriction on our frames: $v \in \bigcup \mathcal{N}_{w}$ $\Rightarrow$ $\bigcup \mathcal{N}_{v} \subseteq \bigcup \mathcal{N}_{w}$. In a bi-relational setting we could say that this is just transitivity of modal visibility (see [5] for comparison).

\item \textbf{B}: $\varphi \rightarrow \lnot \Delta \lnot \Delta \varphi$

\item \textbf{5}: $\lnot \Delta \varphi \rightarrow \Delta \lnot \Delta \varphi$

\item \textbf{GL}: $\Delta(\Delta \varphi \rightarrow \varphi) \rightarrow \Delta \varphi$.

\end{enumerate}

It is easy to find counter-examples. Our logic is rather weak. We can say that that \textbf{IML1} is just \textbf{IPC} with modal axioms \textbf{K}, \textbf{T}, and rules \textbf{MP}, \textbf{RN}. Of course we also hold that \textbf{IML1} is a set of formulas which are true in all those structures which we called \textbf{nIML1}-models. Thus we need completeness theorem.

\section{Canonical model and completeness}

In this section we show how to obtain completeness of system \textbf{IML1} with respect to neighborhood semantics. Canonical frames and models seem to be a natural tool for proving such results. At first, let us introduce some basic definitions and lemmas.

\begin{df}
We define \textbf{IML1}-theory in a standard manner: as a set of well-formed formulas which contains all axioms and is closed under deduction (i.e. under \textit{modus ponens} and modal rule \textbf{RN}).
\end{df}

Attention: later we shall omit symbols \textbf{IML1} and \textbf{nIML1} for convenience. Speaking informally for a moment, the reader should just remember that in this section everything happens with respect to \textbf{IML1} logic and \textbf{nIML1}-models.

\begin{lem} (see [3], Lemma A.1.)

If $w$ is a theory then $\varphi \rightarrow \psi \in w$ $\Leftrightarrow$ $\psi \in v$ for all theories $v$ such that $w \cup \{\varphi\} \subseteq v$.
\end{lem}

\begin{proof}

$\Rightarrow$

This direction is simple. Assume that $\varphi \rightarrow \psi \in w$. Let us consider theory $v$ such that $w \cup \{\varphi\} \subseteq v$. Then $\varphi \in v$ and $\varphi \rightarrow \psi \in v$. By \textit{modus ponens} we get $\psi \in v$.

$\Leftarrow$

Suppose that $v = \{\delta; \varphi \rightarrow \delta \in w\}$. Now we can say that $w \subseteq v$ because: if $\mu \in w$, then $w \vdash \varphi \rightarrow \mu$ (from intuitionistic axiom). Thus $\varphi \rightarrow \mu \in w$ (because theory is by definition deductively closed). So $\mu \in v$. Also $w \cup \{\varphi\} \subseteq v$ because $\varphi \rightarrow \varphi \in w$ and hence $\varphi \in v$. Then $\psi \in v$ so by the very definition of $v$ we have $\varphi \rightarrow \psi \in v$.

\end{proof}

Note that the lemma above can be considered as a semantic version of deduction theorem. It is important, because in modal logic we cannot use typical syntactic form. For example, it is impossible to transform necessity rule into axiom $\varphi \rightarrow \Delta \varphi$. More precisely, acceptance of such claim would lead us to trivial logic in which $\Delta$ would be useless. Of course, our analogue of deduction theorem holds not only for \textbf{IML1}. In fact, it is much more general.

In the next point we introduce the notion of \textit{prime theory}, repeating standard definition from intuitionistic calculus.

\begin{df}
A theory $w$ is said to be \textit{prime} if it satisfies the following conditions:

\begin{enumerate}

\item $\varphi \lor \psi \in w$ $\Leftrightarrow$ $\varphi \in w$ or $\psi \in w$

\item $\bot \notin w$ (i.e. $w$ is consistent)

\end{enumerate}
\end{df}

\begin{lem}
Each consistent theory $w_{\gamma}$ (not containing formula $\gamma$) can be extended to the prime theory $w'_{\gamma}$.
\end{lem}

\begin{proof}
We can extend $w_{\gamma}$ to the relatively maximal theory $w'_{\gamma}$ using standard method for countable languages or even applying Lindenbaum's lemma based on Kuratowski-Zorn lemma. What is important, is to check that this new theory, established in this process, is actually prime. We will omit subscript $\gamma$ if the context provides sufficient clarity.

Suppose that $\varphi \in w'$ or $\psi \in w'$. Then we can use well-known tautologies $\varphi \rightarrow (\varphi \lor \psi)$ and $\psi \rightarrow (\varphi \lor \psi)$ to get our expected result (by \textit{modus ponens}).

This direction was simple but the second one is slightly more complicated. At first, let us introduce the notion of $\overline{X}$ - deductive closure of $X$, where $X$ is a set of formulas. Now assume that $\varphi \lor \psi \in w'$ but $\varphi \notin w'$ and $\psi \notin w'$. Consider $\gamma$ (i.e. formula for which $w'$ is relatively maximal) and then think about $\overline{w' \cup \{\varphi\}}$ and $\overline{w' \cup \{\psi\}}$. We show that $\gamma \in \overline{w' \cup \{\varphi\}}$ (and the same for $\overline{w' \cup \{\psi\}}$).

Suppose the opposite: that $\gamma \notin \overline{w' \cup \{\varphi\}}$. Then we can consider $\langle Th_{\gamma}, \subseteq \rangle$ - the set of all theories without $\gamma$. Clearly, $w' \subseteq \overline{w' \cup \{\varphi\}}$ and moreover $w' \neq \overline{w' \cup \{\varphi\}}$ (because of $\varphi$, for example). But at the same time $\overline{w' \cup \{\varphi\}} \in Th_{\gamma}$ so $w'$ cannot be maximal in $Th_{\gamma}$ (contradiction). Hence, $\gamma \in w' \cup \{\varphi\}$. Thus $\gamma \in v$ for all theories $v$ such that $\overline{w' \cup \{\varphi\}} \subseteq v$. Of course $w' \cup \{\varphi\} \subseteq \overline{w' \cup \{\varphi\}}$. Now we can use the deduction theorem and say that $\varphi \rightarrow \gamma \in w'$.

The same reasoning can be repeated for $\psi$. Then we use intuitionistic axiom to say that $\varphi \lor \psi \rightarrow \gamma \in w'$ and (by \textit{modus ponens}) we get that $\gamma \in w'$ which is false.

\end{proof}

Now we can introduce the notion of canonical model.

\begin{df}
We define canonical model as a triple $\langle W, \mathcal{N}, V \rangle$ where $W$, $\mathcal{N}$ and $V$ are defined as presented below.

\begin{enumerate}

\item $W$ is the set of all prime theories.

\item Neighborhood function for canonical model is a mapping $\mathcal{N}: W \rightarrow P(P(W))$ such that

$X \in \mathcal{N}_{w}$ $\Leftrightarrow$ $\bigcap \mathcal{N}_{w} \subseteq X \subseteq \bigcup \mathcal{N}_{w}$, where:

\begin{enumerate}

\item $\bigcap \mathcal{N}_{w} = \{v \in W: w \subseteq v\}$.

\item $\bigcup \mathcal{N}_{w} = \{v \in W: \Delta \varphi \in w \Rightarrow \varphi \in v\}$.
\end{enumerate}

\item $V$ is a valuation defined in standard manner, i.e. as a function $V: PV \rightarrow P(W)$ such that $w \in V(q)$ $\Leftrightarrow$ $q \in w$.

\end{enumerate}

\end{df}

Note that minimal $w$-neighborhood is just a collection of all supertheories of given theory $w$. In case of maximal $w$-neighborhood we simply take everything what we need, saying colloquially. As for the valuation, it satisfies our expected intuitionistic condition: if $w \in V(q)$ then $\bigcap \mathcal{N}_{w} \subseteq V(q)$.

\begin{lem}
Canonical neighborhood model is defined in a proper way.
\end{lem}

\begin{proof}

We must check conditions from the definition of frame:

\begin{enumerate}

\item $w \in \bigcap \mathcal{N}_{w}$. This restriction holds by the very definition of $\bigcap \mathcal{N}_{w}$ as a collection of supersets of theories.

\item $\bigcap \mathcal{N}_{w} \in \mathcal{N}_{w}$. The answer is positive because $\bigcap \mathcal{N}_{w} \subseteq \bigcap \mathcal{N}_{w}$ and $\bigcap \mathcal{N}_{w} \subseteq \bigcup \mathcal{N}_{w}$. The first inclusion is obvious, so we must check the second one. Suppose that $z \in \bigcap \mathcal{N}_{w}$. Then (as we shall show below) $\Delta$-condition holds, i.e. $\bigcup \mathcal{N}_{z} \subseteq \bigcup \mathcal{N}_{w}$. Now assume that $\Delta \varphi \in z$. Recall that $z$ is theory so it contains modal axiom \textbf{T} which says that $\Delta \varphi \rightarrow \varphi$. Thus (by \textit{modus ponens}) $\varphi \in z$. Then $z \in \bigcup \mathcal{N}_{z}$ which gives us that $z \in \bigcup \mathcal{N}_{w}$.

\item $\bigcap \mathcal{N}_{w} \subseteq X \subseteq \bigcup \mathcal{N}_{w}$ $\Rightarrow$ $X \in \mathcal{N}_{w}$. Relative superset axiom holds by the very definition of canonical neighborhood function.

\item $v \in \bigcap \mathcal{N}_{w}$ $\Rightarrow$ $\bigcap \mathcal{N}_{v} \subseteq \bigcap \mathcal{N}_{w}$. This condition also holds. Suppose that $v \in \bigcap \mathcal{N}_{w}$. Now take $u \in \bigcap \mathcal{N}_{v}$. Then $v \subseteq u$ but also $w \subseteq v$ - so $w \subseteq u$. Hence, $u \in \bigcap \mathcal{N}_{w}$.

\item $v \in \bigcap \mathcal{N}_{w}$ $\Rightarrow$ $\bigcup \mathcal{N}_{v} \subseteq \bigcup \mathcal{N}_{w}$. Suppose the opposite: that there are particular $w, v \in W$ such that $v \in \bigcap \mathcal{N}_{w}$ but $\bigcup \mathcal{N}_{v} \nsubseteq \bigcup \mathcal{N}_{w}$. So we have $u \in \bigcup \mathcal{N}_{v}$ which does not belong to $\bigcup \mathcal{N}_{w}$. Thus, if $\Delta \varphi \in v$ then $\varphi \in u$ but at the same time we can find certain $\psi$ such that $\Delta \psi \in w$ and $\psi \notin u$. But if $\Delta \psi \in w$, then $\Delta \psi \in v$ (because $w \subseteq v$). Hence, $\psi \in u$ (contradiction).

\end{enumerate}

\end{proof}

Now we must show the so-called \textit{fundamental theorem}, typical for any proof of canonical completeness.

\begin{tw}
If $\langle W, \mathcal{N}, V \rangle$ is a canonical model (defined as above) then for each $w \in W$ we have the following equivalence: $w \Vdash \varphi$ $\Leftrightarrow$ $\varphi \in w$.

\end{tw}

\begin{proof}
The proof goes by the induction over the complexity of formulas. We have two really important cases: $\rightarrow$ and $\Delta$.

\begin{enumerate}

\item Consider $\varphi := \gamma \rightarrow \psi$.

$\Rightarrow$

By contraposition assume that $\gamma \rightarrow \psi \notin w$. Deduction theorem and extension lemma allow us to say that there exists $v \in W$ such that $\gamma \in v, w \subseteq v$ and $\psi \notin v$. But if $w \subseteq v$, then $v \in \bigcap \mathcal{N}_{w}$. By induction hypothesis, $v \Vdash \gamma$ and $v \nVdash \psi$. Hence, $w \nVdash \gamma \rightarrow \psi$.

$\Leftarrow$

Suppose that $\gamma \rightarrow \psi \in w$. Assume that $\bigcap \mathcal{N}_{w} \subseteq \{v \in W: v \Vdash \gamma\}$. By induction $\gamma \in w$. Now take arbitrary $u \in \bigcap \mathcal{N}_{w}$. By the definition of minimal $w$-neighborhood we have $w \subseteq u$. Thus $\gamma \rightarrow \psi \in u$. This gives us that $\gamma \in u$ and $\gamma \rightarrow \psi \in u$. By \textit{modus ponens} $\psi \in u$ so we can say that $\bigcap \mathcal{N}_{w} \subseteq \{v \in W: v \Vdash \gamma \Rightarrow v \Vdash \psi\}$. Thus $w \Vdash \gamma \rightarrow \psi$.

\item Consider $\varphi := \Delta \psi$

$\Rightarrow$

Suppose that $\Delta \psi \notin w$. Let us think about $\bigcup \mathcal{N}_{w} = \{v \in W: \Delta \psi \in w \Rightarrow \psi \in v\}$. If $\Delta \psi \notin w$ then $\psi \notin \{\gamma: \Delta \gamma \in w\}$. We designate this last set as $\Delta^{-1} w$. It is, in particular, a theory of \textbf{IML1}. Check this: let $\delta \rightarrow \xi \in \Delta^{-1} w$ and $\delta \in \Delta^{-1} w$. Then $\Delta(\delta \rightarrow \xi) \in w$ and $\Delta \delta \in w$. Because of \textbf{K} axiom we can say that $\Delta \xi \in w$ - so $\xi \in \Delta^{-1}w$. We have shown that $\Delta^{-1}$ is closed on \textit{modus ponens}. Now check \textbf{RN} rule. Let $\xi \in \Delta^{-1}w$. This gives us that $\Delta \xi \in w$. But $w$ is already a theory, so it holds \textbf{RN}. Thus $\Delta \Delta \xi \in w$. From this $\Delta \xi \in \Delta^{-1}w$.

Now we can go back to the main part of the proof. By extension lemma we can find prime theory $u$ such that $\Delta^{-1} w \subseteq u$ and $\psi \notin u$. Moreover, $u \in \bigcup \mathcal{N}_{w}$ - because if not, then there would be formula $\gamma$ such that $\Delta \gamma \in w$ and $\gamma \notin u$. But this would mean that $\gamma \in \Delta^{-1}w$ and then $\gamma \in u$ (because, as we said, $\Delta^{-1}w \subseteq u$).

Finally, we have established $u \in \bigcup \mathcal{N}_{u}$ such that $\psi \notin u$ which means - by induction hypothesis - that $u \nVdash \varphi$. Thus $w \nVdash \Delta \psi$, i.e. $w \nVdash \varphi$.

$\Leftarrow$

Suppose that $w \nVdash \Delta \psi$. So there exists $v \in \bigcup \mathcal{N}_{w}$ such that $v \nVdash \psi$. By induction hypothesis $\psi \notin v$. It follows that $\Delta \psi \notin w$. Suppose the opposite: if $\Delta \psi \in w$ then from the definition of $\bigcup \mathcal{N}_{w}$ we have $\psi \in v$.

\end{enumerate}

\end{proof}

\textit{(Completeness)}
Now our method is standard. Suppose that $w$ is a theory and $w \nvdash \varphi$. In particular this means that $\varphi \notin w$. Then we can extend $w$ to the relatively prime theory $v$ such that $w \subseteq v$ and $\varphi \notin v$. Then for each $\psi \in w$, $v \Vdash \psi$ and $v \nVdash \varphi$. The last statement means in particular that $\varphi$ is not a semantical consequence of $w$.

\section{Bi-relational point of view}

\subsection{Structure and model}
Although we believe that \textit{neighborhood language} is quite comfortable to speak about intuitionistic modal logic (especially if we are interested mainly in minimal and maximal neighborhoods), it is possible to treat our models as bi-relational ones. In fact, this approach is very typical tool for mathematicians working upon modal logics based on intuitionistic core.

It is not difficult to imagine how should the proper bi-relational \textbf{IML1}-semantics look like. Below we present essential formalisation.

\begin{df}
Bi-relational \textbf{IML1}-frame (\textbf{brIML1}-frame) is a triple $F^{\leq, R} = \langle W, \leq, R \rangle$ where $W$ is a non-empty set (of worlds), $\leq$ is a partial order on $W \times W$ and $R$ is a binary relation on $W \times W$ such that:

\begin{enumerate}
\item $\forall {w, v \in W}$ $w \leq v \Rightarrow w R v$
\item $\forall {w, v, u \in W}$ $w \leq v \Rightarrow \left( v R u \Rightarrow w R u \right)$

\end{enumerate}
\end{df}

Alternatively, we can say that $R$ is reflexive and clause 2) holds. The reader can suspect that the second clause is analogous to the $\Delta$-condition in neighborhood frame. Of course this is true, as we shall see later. Now let us define the notion of \textbf{brIML1}-model.

\begin{df}
Bi-relational \textbf{IML1}-model (\textbf{brIML1}-model) is a quadruple $M^{\leq, R} = \langle W, \leq, R, V \rangle$ where $\langle W, \leq, R \rangle$ is a \textbf{brIML1}-frame and $V$ is a function from $PV$ into $P(W)$ satisfying the following condition: if $w \in V(q)$ then for each $v \in W$ such that $w \leq v$ we have $v \in V(q)$.
\end{df}

\begin{df}
Forcing of formulas in a world $w \in W$ is defined inductively in a manner typical for relational (Kripke) models. The only case which is untypical and requires use of relation $R$ is the modal case.

$w \Vdash \Delta \varphi$ $\Leftrightarrow$ $\forall {v \in W}$ $w R v \Rightarrow v \Vdash \varphi$

\end{df}

\begin{lem} Monotonicity of forcing holds in $M^{\leq, R}$, i.e. if $w \Vdash \varphi$ then $\forall {v \in W}$ $w \leq v \Rightarrow v \Vdash \varphi$.
\end{lem}

\begin{proof}
The proof goes by induction and the only case which is important, is modal case. Suppose that $w \Vdash \Delta \psi$. Take $v \in W$ such that $w \leq v$. Now suppose that $v R u$. It means that $w R u$. But if $\Delta \psi$ is accepted in $w$ (as we assumed) then $\psi$ is accepted in each $x \in W$ such that $w R x$ (by the definition of forcing). Thus, $u \Vdash \psi$ and then $v \Vdash \Delta \varphi$.
\end{proof}

As we could see, clause 2 from the definition of \textbf{brIML1}-frame was crucial for this proof - just like $\Delta$-condition was crucial for the analogous proof in neighborhood environment.

In fact, our class of \textbf{brIML1}-frames (models) is identical with a class of \textit{condensed} \textbf{H} $\Box$ frames (models), introduced by Bozic and Dosen in [1]. They have obtained completeness result but (basically) with respect to slightly different axiomatic. Precisely, they used modal axioms $\Delta \varphi \land \Delta \psi \rightarrow \Delta (\varphi \land \psi)$, $\Delta (\varphi \rightarrow \varphi)$ and the rule $\varphi \rightarrow \psi \Vdash \Delta \varphi \rightarrow \Delta \psi$. To avoid ambiguity, we present those theorems with $\Delta$, although it is clear that they used typical $\Box$ symbol.

Now we can establish translations between neighborhood and bi-relational models. In other words, we show that it is possible to change our point of view, preserving frame properties and forcing of formulas. The fact that the same formulas are true in every world of each model based on given universe is named \textit{pointwise equivalency}.

\subsection{From bi-relational to neighborhood model}

\begin{tw}
Suppose that $M^{\leq, R} = \langle W, \leq, R, V \rangle$ is a \textbf{brIML1}-model. Then there exists neighborhood model $M^{\mathcal{N}} = \langle W, \mathcal{N}, V \rangle$ such that both models are pointwise equivalent.
\end{tw}

\begin{proof}
Let us define (for each $w \in W$) two sets:

\begin{enumerate}

\item $\bigcap \mathcal{N}_{w} = \{v \in W; w \leq v\}$
\item $\bigcup \mathcal{N}_{w} = \{v \in W; w R v\}$

\end{enumerate}

Now define function $\mathcal{N}: W \rightarrow P(P(W))$ such that $X \in \mathcal{N}_{w}$ $\Leftrightarrow$ $\bigcap \mathcal{N}_{w} \subseteq X \subseteq \bigcup \mathcal{N}_{w}$. We claim that $\langle W, \mathcal{N}, V \rangle$ is a \textbf{nIML1}-model. Let us check five properties of neighborhood frame:

\begin{enumerate}

\item $w \in \bigcap \mathcal{N}_{w}$ - this condition is satisfied just because $w \leq w$.

\item $\bigcap \mathcal{N}_{w} \in \mathcal{N}_{w}$ - this condition is satisfied. Of course we must show that $\bigcap \mathcal{N}_{w} \subseteq \bigcup \mathcal{N}_{w}$. Suppose that $u \in \bigcap \mathcal{N}_{w}$. It means that $w \leq u$. But then - by the definition of $R$ - $w R u$. Thus $u \in \bigcup \mathcal{N}_{w}$.

\item $u \in \bigcap \mathcal{N}_{w}$ $\Rightarrow$ $\bigcap \mathcal{N}_{u} \subseteq \bigcap \mathcal{N}_{w}$ - this is true. Suppose that $u \in \bigcap \mathcal{N}_{w}$. Thus $w \leq u$. Now if certain $v \in \bigcap \mathcal{N}_{u}$, then $w \leq u \leq v$. Hence, $v \in \bigcap \mathcal{N}_{w}$.

\item $\bigcap \mathcal{N}_{w} \subseteq X \subseteq \bigcup \mathcal{N}_{w}$ $\Rightarrow$ $X \in \bigcap \mathcal{N}_{w}$ - this is obvious (by the very definition of $\mathcal{N}$).

\item $u \in \bigcap \mathcal{N}_{w}$ $\Rightarrow$ $\bigcup \mathcal{N}_{u} \subseteq \bigcup \mathcal{N}_{w}$ - this is $\Delta$-condition. Assume that $u \in \bigcap \mathcal{N}_{w}$, i.e. $w \leq u$. Let $v \in \bigcup \mathcal{N}_{u}$ which means that $u R v$. Then, by the definition of $R$, we have also that $w R v$. Hence, $v \in \bigcup \mathcal{N}_{w}$.

\end{enumerate}

Pointwise equivalency can be proved by induction on the construction of formula. In fact, it is quite obvious if we remember about strict correspondence between minimal (maximal) neighborhoods and relations $\leq, R$. Moreover, we did not change valuation, preparing neighborhood model, so in each world the same propositional variables are accepted.

\end{proof}

\subsection{From neighborhood model to the bi-relational one}

\begin{tw}
Suppose that $M^{\mathcal{N}} = \langle W, \mathcal{N}, V \rangle$ is an \textbf{brIML1} neighborhood model. Then there exists bi-relational model $M^{\leq, R} = \langle W, \leq, R, V \rangle$ such that both models are pointwise equivalent.
\end{tw}

\begin{proof}
Let us define two relations on $W \times W$, $\leq$ and $R$:

\begin{enumerate}
\item $w \leq v$ $\Leftrightarrow$ $v \in \bigcap \mathcal{N}_{w}$

\item $w R v$ $\Leftrightarrow$ $v \in \bigcup \mathcal{N}_{w}$

\end{enumerate}

It is easy to show that $\leq$ is actually a preorder. Let us check two conditions which combine $\leq$ with $R$.

\begin{enumerate}

\item $\forall {w, v \in W}$ $w \leq v \Rightarrow w R v$ - this condition is satisfied because minimal $w$-neighborhood is contained in the maximal one.

\item $\forall {w, v, u \in W}$ $w \leq v \Rightarrow \left( v R u \Rightarrow w R u \right)$ - this condition is satisfied because of $\Delta$-condition. Suppose that $w \leq v$ and $v R u$. Thus $v \in \bigcap \mathcal{N}_{w}$ and $u \in \bigcup \mathcal{N}_{v}$. But $\bigcup \mathcal{N}_{v} \subseteq \bigcup \mathcal{N}_{w}$ so $w R u$.

\end{enumerate}

The last thing is pointwise equivalency. The proof goes by induction on the complexity of formulas.

\end{proof}

\subsection{Finite model property}

In this section we show that \textbf{IML1} logic has finite model property. We achieve this result by means of filtration, working in bi-relational framework (although it would be possible to use neighborhood environment). Check Hashimoto [2] and Takano [10] to study other examples of such approach in intuitionistic modal logic. Attention: we omit superscript ${}^{\leq, R}$ later in this subsection, when speaking about bi-relational frames.

\begin{df}
Consider \textbf{brIML1}-model $M = \langle W, \leq, R, V \rangle$ and formula $\gamma$ which is not tautology, i.e. there exists $w \in W$ such that $w \nVdash \gamma$. Let us define: $\Sigma = Sub(\gamma)$, which means that $\Sigma$ is the set of all subformulas of $\gamma$.
\end{df}

\begin{df}
We define the equivalence relation $\sim$ on $W$ as follows: $w \sim v$ iff $w \Vdash \alpha$ $\Leftrightarrow$ $v \Vdash \alpha$ for every $\alpha \in \Sigma$. We denote by $[w]$ the equivalence class of an element $w$.
\end{df}

\begin{df}
Filtered \textbf{brIML1}-frame $F_{\Sigma} = \langle W, \leq_{\Sigma}, R_{\Sigma} \rangle$ is defined as follows:

\begin{enumerate}

\item $W_{\Sigma} = \{ [w]; w \in W\}$

\item $[w] \leq_{\Sigma} [v]$ $\Leftrightarrow$ if $w \Vdash \alpha$, then $v \Vdash \alpha$ for each $\alpha \in \Sigma$

\item $[w] R_{\Sigma} [v]$ $\Leftrightarrow$: if $w \Vdash \Delta \beta$, then $v \Vdash \beta$ for each $\Delta \beta \in \Sigma$

\end{enumerate}

\end{df}

Of course $W_{\Sigma}$ is finite (because $\Sigma$ is finite). But we must assure that we obtained proper \textbf{brIML1}-structure. Fortunately, we have the following theorem:

\begin{tw}
Filtered frame $F_{\Sigma}$, defined as above, is a well-defined \textbf{brIML1}-frame.
\end{tw}

\begin{proof}
We must check three conditions:

\begin{enumerate}

\item $\leq_{\Sigma}$ is pre-order. This is almost obvious.

\item $[w] \leq_{\Sigma} [v] \Rightarrow [w] R_{\Sigma} [v]$. Suppose that $[w] \leq_{\Sigma} [v]$ and $w \Vdash \Delta \beta$ where $\Delta \beta \in \Sigma$. Thus $v \Vdash \Delta \beta$. But then $v \Vdash \beta$ which gives us our expected result.

\item $[w] \leq_{\Sigma} [v] \Rightarrow ([v] R_{\Sigma} [u] \Rightarrow [w] R_{\Sigma} [u])$. Suppose that $[w] \leq_{\Sigma} [v]$ and $[v] R_{\Sigma} [u]$. Now assume that $w \Vdash \Delta \beta, \Delta \beta \in \Sigma$. Thus $v \Vdash \Delta \beta$. Hence, $u \Vdash \beta$.

\end{enumerate}

\end{proof}

The next lemma is clear:

\begin{lem}
The following properties are true for every $w, v \in W$:

\begin{enumerate}

\item If $w \leq v$, then $[w] \leq_{\Sigma} [v]$.

\item If $w R v$, then $[w] R_{\Sigma} [v]$.

\end{enumerate}
\end{lem}

Now we can transform our filtered structure into an \textbf{brIML1}-model - by introducing valuation.

\begin{df}

Let $M = \langle W, \leq, R, V \rangle$ be an \textbf{brIML1}-model. Suppose that $F_{\Sigma} = \langle W_{\Sigma}, \leq_{\Sigma}, R_{\Sigma} \rangle$ is a filtered \textbf{brIML1}-frame based on the structure of $M$. We define filtered \textbf{brIML1}-model $M_{\Sigma}$ as a quadruple $\langle W, \leq_{\Sigma}, R_{\Sigma}, V_{\Sigma} \rangle$, where $V_{\Sigma}$ is a function from $PV$ into $P(W_{\Sigma})$ such that:

$$V_{\Sigma}(q) = \begin{cases}
      \{[w]; w \in V(q)\} & q \in \Sigma \\
      \emptyset & q \notin \Sigma
   \end{cases}
$$

\end{df}

It is easy to check that $V_{\Sigma}$ is monotone with respect to $\leq_{\Sigma}$ (for propositional variables and then, by induction, for all formulas). Now we can prove the crucial theorem:

\begin{tw}
The following equivalence holds for each $w \in W$ and every $\alpha \in \Sigma$:

$w \Vdash \alpha$ $\Leftrightarrow$ $[w] \Vdash_{M_{\Sigma}} \alpha$.
\end{tw}

\begin{proof}
The proof goes by the standard induction on the construction of $\alpha$. Let us check the most important cases:

\begin{enumerate}

\item $\alpha := \varphi \rightarrow \psi$. Suppose that $w \Vdash \alpha$. Thus for all $v \in W, w \leq v$, we have $v \nVdash \varphi$ or $v \Vdash \psi$. Hence, by induction hypothesis, $[v] \nVdash_{M_{\Sigma}} \varphi$ or $[v] \Vdash_{M_{\Sigma}} \psi$. Of course $[w] \leq [v]$. Now we can say that $[w] \Vdash_{M_{\Sigma}} \varphi \rightarrow \psi$.

    On the other side, assume that $[w] \Vdash_{M_{\Sigma}} \alpha$. Thus, for each $[v] \in W_{\Sigma}$ such that $[w] \leq_{\Sigma} [v]$, we have $[v] \nVdash_{M_{\Sigma}} \varphi$ or $[v] \Vdash_{M_{\Sigma}} \psi$. By induction hypothesis $v \nVdash \varphi$ or $v \Vdash \psi$. As we know, $[w] \leq_{\Sigma} [v]$ means that acceptation of an arbitrary formula from $\Sigma$ in $w$ implies its acceptation in $v$. In particular, this condition holds when $w \leq v$. For this reason, we can write that $w \Vdash \alpha$.

\item $\alpha := \Delta \varphi$. Assume that $w \Vdash \alpha$. Now for all $v \in W, w R v$, we have $v \Vdash \varphi$. Thus, by induction hypothesis, $[v] \Vdash_{M_{\Sigma}} \beta$. Clearly, $[w] R_{\Sigma} [v]$ - so $[w] \Vdash_{M_{\Sigma}} \alpha$.

    On the other side, suppose that $[w] \Vdash_{M_{\Sigma}} \alpha$. Thus, for each $[v] \in W_{\Sigma}$ such that $[w] R_{\Sigma} [v]$, we have $[v] \Vdash_{M_{\Sigma}} \varphi$. By induction hypothesis, $v \Vdash \varphi$. As we know, $[w] R_{\Sigma} [v]$ means that acceptation of any formula $\Delta \beta$ from $\Sigma$ in $w$ implies acceptation of $\beta$ in $v$. In particular, this condition holds when $w R v$. For this reason, we can write that $w \Vdash \alpha$.

\end{enumerate}

\end{proof}

Finally, we can sum up our investigations in the conclusion below:

\begin{tw}
Intuitonistic modal logic \textbf{IML1} has finite model property and thus is decidable. In other words, it means that mono-modal intuitionistic system \textbf{KT} has finite model property with respect to the class of bi-relational \textbf{brIML1}-models.
\end{tw}

Of course the same conclusion holds for \textbf{nNML1}-models.

\section{Additional connectives}

Here we introduce functors announced earlier. Surely we have $\bot$, so $\sim$ can be derived from $\rightsquigarrow$ - but for transparency we introduce it also directly.

$w \Vdash \varphi \rightsquigarrow \psi$ $\Leftrightarrow$ $\bigcup \mathcal{N}_{w} \subseteq \{v \in W; v \nVdash \varphi$ or $v \Vdash \psi\}$

$w \Vdash \sim \varphi$ $\Leftrightarrow$ $\bigcup \mathcal{N}_{w} \subseteq \{v \in W; v \nVdash \varphi\}$

Theorem 2.1. (about monotonicity of forcing) holds for $\sim \varphi$ and $\varphi \rightsquigarrow \psi$. Now one could ask if we actually need new functors.

\begin{lem}

We can eliminate $\sim$ and $\rightsquigarrow$ if we accept the following restriction:

$v \in \bigcup \mathcal{N}_{w}$ $\Rightarrow$ $\bigcap \mathcal{N}_{v} \subseteq \bigcup \mathcal{N}_{w}$ ($T$-condition)

More precisely, in the presence of $T$-condition we can say that: $w \Vdash \varphi \rightsquigarrow \psi$ $\Leftrightarrow$ $w \Vdash \Delta(\varphi \rightarrow \psi)$ (for every $w \in W$).

\end{lem}

\begin{proof}
$\Rightarrow$

Suppose that there exists \textbf{nIML1}-model with $T$-condition such that for certain $w \in W$ we have: $w \Vdash \varphi \rightsquigarrow \psi$ but $w \nVdash \Delta(\varphi \rightarrow \psi)$. It means that for each $v \in \bigcup \mathcal{N}_{w}$, $v \nVdash \varphi$ or $v \Vdash \psi$. But at the same time there exists $u \in \bigcup \mathcal{N}_{w}$ such that $u \nVdash \varphi \rightarrow \psi$. So there is $z \in \bigcap \mathcal{N}_{u}$ such that $z \Vdash \varphi$ and $z \nVdash \psi$. But this is impossible because $\bigcap \mathcal{N}_{u} \subseteq \bigcup \mathcal{N}_{w}$.

$\Leftarrow$

This direction is obvious: it is clear that if we have $\varphi \rightarrow \psi$ in each point of $\bigcup \mathcal{N}_{w}$ then we have $\varphi \rightsquigarrow \psi$ at least in $w$.

\end{proof}

Now we can introduce the notion of \textbf{nIML2}-model (i.e. \textbf{nIML1} with $T$-condition). Note that we can say that \textbf{nIML2}-models are distinguished from \textbf{nIML1}-models by the formula $\varphi \rightsquigarrow \psi \rightarrow \Delta (\varphi \rightarrow \psi)$ which holds in the former structures and not in the latter. Alternatively, we can use $\sim \varphi \rightarrow \Delta \lnot \varphi$.

Note the following lemma:

\begin{lem}Canonical \textbf{nIML1}-model validates $T$-condition (without introducing any new axioms on the syntactical side)\end{lem}

\begin{proof}
Assume that we have (prime) theories $v$ and $w$ such that $v \in \bigcup \mathcal{N}_{w}$. It means that if $\Delta \varphi$ belongs to $w$, then $\varphi$ is in $v$ (for every formula $\varphi$). Suppose now that $\bigcap \mathcal{N}_{v} \nsubseteq \bigcup \mathcal{N}_{w}$. Thus, we have certain $u$ such that $u \in \bigcap \mathcal{N}_{v}$ and $u \notin \bigcup \mathcal{N}_{w}$. Hence, $v \subseteq u$ and we can find at least one formula $\gamma$ such that $\Delta \gamma \in w$ but $\gamma \notin u$. On the other hand, $\gamma \in v$. But $v \subseteq u$, as we said - so we have clear contradiction.
\end{proof}

The last lemma is important. In fact, it says that \textbf{IML1} system is complete with respect to \textbf{nIML2}-frames. Moreover, similar reasoning (in terms of bi-relational structures) can be repeated in filtered model. Thus, we have finite model property w.r.t. \textbf{brIML2}-frames.

It should be added that Bozic and Dosen also distinguished $T$-condition (and proved completeness directly in bi-relational setting). In their terminology, \textbf{brIML2}-frames are \textit{strictly condensed} frames.

\section{Bisimulation and congenial notions}

The notions of bounded morphism and bisimulation are well-known and widely used by researchers working with relational (Kripke) or neighborhood models. In this chapter we show how those tools can be used in our environment (see also [4] for comparison). Note that in the fourth subsection our new connective $\rightsquigarrow$ becomes really necessary.

\subsection{Bounded morphisms}

\begin{df}
Let $M^1 = \langle W^1, \mathcal{N}^{1}, V^1 \rangle$ and $M^2 = \langle W^2, \mathcal{N}^{2}, V^2 \rangle$ be two \textbf{nIML1}-models. A function $f: W^1 \rightarrow W^2$ is a \textit{bounded morphism} from $M^1$ to $M^2$ if for any $w \in W$ these restrictions are satisfied:

\begin{enumerate}

\item for every $q \in PV$ we have $w \Vdash_{M^1} q$ $\Leftrightarrow$ $f(w) \Vdash_{M^2} q$. This means that $w$ and $f(w)$ satisfy the same propositional variables.

\item $f[\bigcap \mathcal{N}^{1}_{w}] = \bigcap \mathcal{N}^{2}_{f(w)}$

\item $f[\bigcup \mathcal{N}^{1}_{w}] = \bigcup \mathcal{N}^{2}_{f(w)}$

\end{enumerate}

\end{df}

Now we can prove the following theorem:

\begin{tw}
Let $M^1 = \langle W^1, \mathcal{N}^{1}, V^1 \rangle$ and $M^2 = \langle W^2, \mathcal{N}^{2}, V^2 \rangle$ be two \textbf{nIML1}-models. If $f: W^1 \rightarrow W^2$ is a bounded morphism from $M^1$ to $M^2$ then for each formula $\varphi$ and each $w \in W^1$ the following holds: $w \Vdash_{M^1} \gamma$ $\Leftrightarrow$ $f(w) \Vdash_{M^2} \gamma$.
\end{tw}

\begin{proof}
The proof goes by induction on the complexity of formulas. The first case (propositional letters) is obvious. Also $\land$ and $\lor$ are quite simple. Implication cases ($\rightarrow$ and $\rightsquigarrow$) are only slightly more complicated. In general, we can repeat the reasoning of Moniri and Maleki.

\begin{enumerate}

\item $\gamma := \varphi \rightarrow \psi$

It is sufficient to show that $\bigcap \mathcal{N}^{1}_{w} \subseteq \{v \in W^1; v \nVdash_{M^1} \varphi$ or $v \Vdash_{M^1} \psi\}$  $\Leftrightarrow$ $\bigcap \mathcal{N}^{2}_{f(w)} \subseteq \{u \in W^2; u \nVdash_{M^2} \varphi$ or $u \Vdash_{M^2} \psi\}$.

$\Rightarrow$
Assume that $y \in \bigcap \mathcal{N}^{2}_{f(w)}$. Hence there exists $x \in \bigcap \mathcal{N}^{1}_{w}$ such that $f(x) = y$. Then $x \nVdash_{M^1} \varphi$ or $x \Vdash_{M^1} \psi$. By induction hypothesis, $y \nVdash_{M^2} \varphi$ or $y \Vdash_{M^2} \psi$. Clearly, this gives us our expected conclusion because $y$ is an arbitrary element of $\bigcap \mathcal{N}^{2}_{f(w)}$.

$\Leftarrow$
This direction is similar and it starts with an assumption that $x \in \bigcap \mathcal{N}^{1}_{w}$.

\item $\gamma := \varphi \rightsquigarrow \psi$

Here we can repeat reasoning for $\rightarrow$ but of course we must remember that now we are working with maximal (not minimal) neighborhoods.

\item $\gamma := \Delta \varphi$

Assume that $w \Vdash_{M^1} \Delta \varphi$. Hence $\bigcup \mathcal{N}^{1}_{w} \subseteq \{v \in W^1; v \Vdash_{M^1} \varphi\}$. We want to show that $\bigcup \mathcal{N}^{2}_{f(w)} \subseteq \{z \in W^2; z \Vdash_{M^2} \varphi\}$. Take $y \in \bigcup \mathcal{N}^{2}_{f(w)}$. Then there is $x \in \bigcup \mathcal{N}^{1}_{w}$ such that $f(x) = y$. But if $w \Vdash_{M^1} \Delta \varphi$ then $x \Vdash_{M^1} \varphi$. Thus, by induction hypothesis we have $f(x) = y \Vdash_{M^2} \varphi$. Now we can say that $f(w) \Vdash_{M^2} \Delta \varphi$.

Just as in the case of implication, the other direction is very similar. We suppose that $f(w) \Vdash_{M^2} \Delta \varphi$ and then we must show that $\bigcup \mathcal{N}^{1}_{w} \subseteq \{v \in W^1; v \Vdash_{M^1} \varphi\}$.

\end{enumerate}

\end{proof}

\subsection{Behavioral equivalence}

There is one interesting notion related to bounded morphisms (see [6]). We define it below:

\begin{df}
Suppose that we have two \textbf{nIML1}-models $M^1 = \langle W^1, \mathcal{N}^{1}, V^1 \rangle$ and $M^2 = \langle W^2, \mathcal{N}^{2}, V^2 \rangle$. We say that worlds $w_1 \in W^1$ and $w_2 \in W^2$ are behaviorally equivalent if there exists \textbf{nIML1}-model $B = \langle W_B, \mathcal{N}_{B}, V_B \rangle$ and bounded morphisms $f: M^1 \rightarrow B$ and $g: M^2 \rightarrow B$ such that $f(w_1) = g(w_2)$.
\end{df}

\begin{figure}[ht]
\centering
\includegraphics[height=5cm]{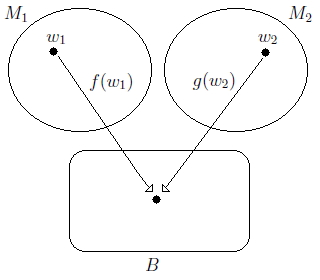}
\caption{Behavioral equivalence}
\label{fig:obrazek {behav}}
\end{figure}

This idea has been depicted on Fig. 2. It is not surprising that the following theorem holds:

\begin{tw}Assume that $M^1 = \langle W^1, \mathcal{N}^{1}, V^1 \rangle$ and $M^2 = \langle W^2, \mathcal{N}^{2}, V^2 \rangle$ are two \textbf{nIML1}-models. If $w_1 \in W^1$ and $w_2 \in W^2$ are behaviorally equivalent, then for each $\varphi$ we have: $w_1 \Vdash_{M^1} \varphi$ $\Leftrightarrow$ $w_2 \Vdash_{M^2} \varphi$.
\end{tw}

\begin{proof}
Suppose that $w_1 \Vdash_{M^1} \varphi$. Consider model $B$ and $f: M^1 \rightarrow B$. From Theorem 7.1. we can say that $f(w_1) \Vdash_{B} \varphi$. But $f(w_1) = g(w_2)$. Thus $g(w_2) \Vdash_B \varphi$. Now we use the same reasoning to get our conclusion: that $w_2 \Vdash_{M_2} \varphi$.
\end{proof}

\subsection{Bisimulation}

\begin{df}

Let $ M^1 = \langle W^1, \mathcal{N}^{1}, V^1 \rangle $ and $M^2 = \langle W^2, \mathcal{N}^{2}, V^2 \rangle$ be two \textbf{nIML1}-models. A non-empty binary relation $R \subseteq W^1 \times W^2$ is called a \textit{bisimulation} between $M^1$ and $M^2$ if for all $w_1 \in W^1, w_2 \in W^2$ such that $w_1 R w_2$ we have:

\begin{enumerate}

\item For every $q \in PV$: $w_1 \Vdash_{M^1} q$ $\Leftrightarrow$ $w_2 \Vdash_{M^2} q$

\item For every $y \in \bigcap \mathcal{N}^{2}_{w_2}$ there is $x \in \bigcap \mathcal{N}^{1}_{w_1}$ such that $xRy$

\item For every $y \in \bigcup \mathcal{N}^{2}_{w_2}$ there is $x \in \bigcup \mathcal{N}^{1}_{w_1}$ such that $xRy$

\item For every $x \in \bigcap \mathcal{N}^{1}_{w_1}$ there is $y \in \bigcap \mathcal{N}^{2}_{w_2}$ such that $xRy$

\item For every $x \in \bigcup \mathcal{N}^{1}_{w_1}$ there is $y \in \bigcup \mathcal{N}^{2}_{w_2}$ such that $xRy$

\end{enumerate}

\end{df}

As we could expect, the following theorem holds:

\begin{tw}
Let $ M^1 = \langle W^1, \mathcal{N}^{1}, V^1 \rangle $ and $M^2 = \langle W^2, \mathcal{N}^{2}, V^2 \rangle$ be two \textbf{nIML1}-models with a bisimulation $R$ between them. Then $w_1 R w_2$ implies that $w_1$ and $w_2$ force the same formulas.
\end{tw}

\begin{proof}
The propositional case is obvious. Also $\land$ and $\lor$ are not difficult. What is important, is to show that our theorem holds for both implications and $\Delta$.

\begin{enumerate}

\item $\gamma := \varphi \rightarrow \psi$

Suppose that $w_1 \Vdash_{M^1} \gamma$. Thus $\bigcap \mathcal{N}^{1}_{w_1} \subseteq \{v \in W^1; v \nVdash \varphi$ or $v \Vdash \psi\}$. Take $y \in \bigcap \mathcal{N}^{2}_{w_2}$ and assume that $y \Vdash_{M^2} \varphi$. Because of bisimulation we can find $x \in \bigcap \mathcal{N}^{1}_{w_1}$ such that $xRy$. Now $x \Vdash_{M^1} \varphi$ (by induction hypothesis) but this implies that $x \Vdash_{M^1} \psi$. Then - again by induction hypothesis - $y \Vdash_{M^2} \psi$. So $w_2 \Vdash_{M^2} \gamma$.

The converse direction (starting from the assumption that $w_2 \Vdash_{M^2} \gamma$) is almost identical. Of course, we must use adequate conditions concerning relationships between minimal neighborhoods of both models.

\item $\gamma := \varphi \rightsquigarrow \psi$

This reasoning is very similar to the one presented above but with respect to maximal (not minimal) neighborhoods.

\item $\gamma: = \Delta \varphi$.

Assume that $w_1 \Vdash_{M^1} \gamma$, i.e. $\bigcup \mathcal{N}^{1}_{w_1} \subseteq \{v \in W^1; v \Vdash_{M^1} \varphi\}$. Now suppose that $y \in \bigcup \mathcal{N}^{2}_{w_2}$. Then there is $x \in \bigcup \mathcal{N}^{1}_{w_1}$ such that $xRy$. If $w_1 \Vdash_{M^1} \Delta \varphi$ then $x \Vdash_{M^1} \varphi$. Hence (by induction hypothesis) $y \Vdash_{M^2} \varphi$. This means that $w_2 \Vdash_{M^2} \Delta \varphi$.

Similar reasoning is adequate if we start from the assumption that $w_2 \Vdash_{M^2} \gamma$.

\end{enumerate}
\end{proof}

\subsection{The notion of n-bisimulation}

This section is probably more complex and not so obvious. At first, we introduce the notion of \textit{the degree of formulas}, following Moniri and Maleki but also expanding their basic definition.

\begin{df}
The degree of formulas is defined as follows:

\begin{enumerate}
\item $deg(q) = 0$ for any $q \in PV$
\item $deg(\bot) = 0$
\item $deg(\varphi \lor \psi) = deg(\varphi \land \psi) = \max(deg(\varphi), deg(\psi))$
\item $deg(\varphi \rightarrow \psi) = deg(\varphi \rightsquigarrow \psi) = 1 + \max(deg(\varphi), deg(\psi))$
\item $deg(\Delta \varphi) = 1 + deg(\varphi)$

\end{enumerate}

\end{df}

\begin{df}

Let $M^1 = \langle W^1, \mathcal{N}^{1}, V^1 \rangle $ and $M^2 = \langle W^2, \mathcal{N}^{2}, V^2 \rangle$ be two \textbf{nIML1}-models. Suppose that $w_1 \in M^1$ and $w_2 \in M^2$. We say that $w_1$ and $w_2$ are \textit{n-bisimilar} if there exists a sequence of binary relations $R_n \subseteq ... \subseteq R_0$ with the following properties (for $i + 1 \leq n$):

\begin{enumerate}

\item $w_1 R_n w_2$

\item If $w_1 R_0 w_2$ then $w_1$ and $w_2$ agree on all propositional letters.

\item If $w_1 R_{i+1} w_2$ then for every $y \in \bigcap \mathcal{N}^{2}_{w_2}$ there exists $x \in \bigcap \mathcal{N}^{1}_{w_1}$ such that $x R_i y$;

\item If $w_1 R_{i+1} w_2$ then for every $y \in \bigcup \mathcal{N}^{2}_{w_2}$ there exists $x \in \bigcup \mathcal{N}^{1}_{w_1}$ such that $x R_i y$;

\item If $w_1 R_{i+1} w_2$ then for every $x \in \bigcap \mathcal{N}^{1}_{w_1}$ there exists $y \in \bigcap \mathcal{N}^{2}_{w_2}$ such that $x R_i y$.

\item If $w_1 R_{i+1} w_2$ then for every $x \in \bigcup \mathcal{N}^{1}_{w_1}$ there exists $y \in \bigcup \mathcal{N}^{2}_{w_2}$ such that $x R_i y$;

\end{enumerate}

\end{df}

The following theorem holds (see [4]):

\begin{tw}
Let $\Pi$ be a finite set of propositional variables. Assume that $M^1$ and $M^2$ are \textbf{nIML1}-models for language of $\Pi$. Then for each $w_1 \in M^1$ and $w_2 \in M^2$ the following conditions are equivalent:

\begin{enumerate}
\item $w_1$ and $w_2$ are $n$-bisimilar

\item $w_1$ and $w_2$ agree on all propositional formulas of degree at most $n$.
\end{enumerate}

\end{tw}

\begin{proof}
$\Rightarrow$

This part of the proof goes by induction on $n$. The case $n = 0$ is obvious. Assume that $w_1$ and $w_2$ are $n$-bisimilar. Consider the set $\Gamma_{w_1}^{n+1}$ of all formulas of degree at most $n+1$ that are satisfied by $w_1$. It is easy to see that this collection is equivalent to a single formula $\Theta_1$ which can be written as $\bigwedge \Gamma_{w_1}^{n+1}$. Note that our set of propositional letters is finite so for any $n$ we have (up to logical equivalence) only finitely many formulas of degree at most $n$.

Analogously, for $w_2$ we can consider $\Theta_2$. Our aim is to show that $\Theta_2$ is satisfied by $w_1$ and $\Theta_1$ is true in $w_2$.

Assume now that $deg(\Theta_1) = n+1$. Then $\Theta_1$ is a Boolean combination of propositional letters and formulas of the form:

- $\gamma := \varphi \rightarrow \psi$ such that $max\{deg(\psi_1), deg(\psi_2)\} = n$,

- $\gamma := \varphi \rightsquigarrow \psi$ such that $max\{deg(\psi_1), deg(\psi_2)\} = n$,

- $\gamma := \Delta \varphi$ where $deg(\varphi) \leq n$.

Clearly, it is sufficient to show that if $\gamma$ has one of the forms mentioned above then $w_1 \Vdash_{M_1} \gamma$ implies that $w_2 \Vdash_{M_2} \gamma$.

\begin{enumerate}

\item Suppose that $w_1 \Vdash_{M^1} \varphi \rightarrow \psi$. We want to show that $\bigcap \mathcal{N}^{2}_{w_2} \subseteq \{v \in W^2; v \Vdash_{M^2} \varphi \Rightarrow v \Vdash_{M^2} \psi\}$. Take $y \in \bigcap \mathcal{N}^{2}_{w_2}$ and suppose that $y \Vdash_{M^2} \varphi$. Because of $n$-bisimulation (clause 3 of Def. 7.5) we can find $x \in \bigcap \mathcal{N}^{1}_{w_1}$ such that $x R_n y$. By assumption $deg(\varphi) \leq n$. Hence (by induction hypothesis) $x \Vdash_{M^1} \varphi$. Then $x \Vdash_{M^1} \psi$ which gives us (again by induction hypothesis) that $y \Vdash_{M^2} \psi$. This is our expected result.

    The converse direction is very similar although we must use clause 5 of Def. 7.5.

\item Similar reasoning as above can be repeated for $\gamma := \varphi \rightsquigarrow \psi$. Of course now we need to use adequate clauses for maximal neighborhoods. So suppose that $w_1 \Vdash_{M^1} \varphi \rightsquigarrow \psi$. We want to obtain the following result: that $\bigcup \mathcal{N}^{2}_{w_2} \subseteq \{v \in W^2; v \Vdash_{M_2} \varphi \Rightarrow v \Vdash_{M_2} \psi\}$. Take $y \in \bigcup \mathcal{N}^{2}_{w_2}$ and suppose that $y \Vdash_{M^2} \varphi$. Now we use clause 4 of Def. 7.5. to find $x \in \bigcup \mathcal{N}^{1}_{w_1}$ such that $x R_n y$. By assumption $deg(\varphi) \leq n$. Hence $x \Vdash_{M^1} \varphi$ and then $x \Vdash_{M^1} \psi$. This gives us that $y \Vdash_{M^2} \psi$.

    The other direction is almost identical. However, me must use clause 6 of Def. 7.5.

\item Suppose that $w_1 \Vdash_{M^1} \Delta \varphi$. It means that $\bigcup \mathcal{N}^{1}_{w_1} \subseteq \{v \in W^{1}; v \Vdash \varphi\}$. Let $y \in \bigcup \mathcal{N}^{2}_{w_2}$. Then (by the definition of $n$-bisimulation, clause 4) there exists $x \in \bigcup \mathcal{N}^{1}_{w_1}$ such that $x R_n y$. If $w_1 \Vdash \Delta \varphi$ then $x \Vdash \varphi$. Then (by the induction hypothesis) $y \Vdash \varphi$. Thus $w_2 \Vdash \Delta \varphi$.

    One can easily see that the converse direction, starting from the assumption that $w_2 \Vdash_{M^2} \Delta \varphi$ - is completely analogous.

\end{enumerate}

$\Leftarrow$

Suppose that $w_1$ and $w_2$ agree on all propositional formulas of degree at most $n$. In general, we repeat the main idea of Moniri and Maleki - but we must adapt it to our additional clauses, connected with maximal neighborhoods.

At first, let us define the following sequence of relations:

$R_n = \{\langle x, y \rangle; x \in W^1, y \in W^2, deg(\varphi) \leq n, x \Vdash_{M^1} \varphi \Leftrightarrow y \Vdash_{M^2} \varphi\}$;

$R_{n-1} = \{\langle x, y \rangle; x \in W_1, y \in W^2, deg(\varphi) \leq n-1, x \Vdash_{M^1} \varphi \Leftrightarrow y \Vdash_{M^2} \varphi\}$;

\vdots

$R_{0} = \{\langle x, y \rangle; x \in W^1, y \in W^2, deg(\varphi) = 0, x \Vdash_{M^1} \varphi \Leftrightarrow y \Vdash_{M^2} \varphi\}$;

We see that $R_{n} \subseteq R_{n-1} \subseteq ... \subseteq R_{0}$, as expected. Now we must check specific conditions of $n$-bisimulation. As for clauses 1 and 2, they are obvious. Clauses 3 and 5 have been presented by Moniri and Maleki. Below we show how to modify their solution so that it would fit to clauses 4 and 6.

Suppose now that $u_{1} R_{i+1} u_{2}$. It means, in particular, that $u_1$ and $u_2$ agree on all propositional formulas of degree at most $i+1$. We must show that for every $y \in \bigcup \mathcal{N}^{2}_{u_2}$ there is $x \in \bigcup \mathcal{N}^{1}_{u_1}$ such that $x R_i y$. Assume the contrary: there exists $y \in \bigcup \mathcal{N}^{2}_{u_2}$ such that for every $x \in \bigcup \mathcal{N}^{1}_{u_1}$ there is formula $\varphi_{x}$, $deg(\varphi_{x}) \leq i$ for which:

$y \Vdash_{M^2} \varphi_{x}$ and $x \nVdash_{M^1} \varphi_{x}$

or

$y \nVdash_{M^2} \varphi_{x}$ and $x \Vdash_{M^1} \varphi_{x}$

Let $\Gamma$ be the set of all such $\varphi_{x}$. This set is (up to equivalence) finite. Now let us consider two subsets of $\Gamma$:

$\Gamma_{0} = \{\varphi_{x} \in \Gamma$ such that $y \Vdash_{M^2} \varphi_{x}$ and $x \nVdash_{M^1} \varphi_{x}\}$

$\Gamma_{1} = \{\varphi_{x} \in \Gamma$ such that $y \nVdash_{M^2} \varphi_{x}$ and $x \Vdash_{M^1} \varphi_{x}\}$

Now we define the following formula (of degree not bigger than $i+1$):

$\Theta  = \begin{cases}
      \bigwedge \Gamma_0, & \Gamma_1 = \emptyset \\
      \bigwedge \Gamma_0 \rightsquigarrow \bigvee \Gamma_1, & \Gamma_0 \neq \emptyset, \Gamma_1 \neq \emptyset \\
      \bigvee \Gamma_1, & \Gamma_0 = \emptyset
   \end{cases}
$

Suppose that $\Gamma_1$ is empty. Then we can say that $y \Vdash_{M^2} \Theta$. Thus $u_2 \nVdash_{M^2} \sim \Theta$. For every $x \in \bigcup \mathcal{N}^{1}_{u_1}$ we have $x \nVdash_{M_1} \Theta$. Hence, $\bigcup \mathcal{N}^{1}_{u_1} \subseteq \{z \in W^1; z \nVdash_{M^1} \Theta\}$. So we can say that $u_1 \Vdash_{M^1} \sim \Theta$. This is contradiction because $deg(\sim \Theta) \leq i+1$.

Now assume that $\Gamma_{0} \neq \emptyset$ and $\Gamma_{1} \neq \emptyset$. Consider

$K = \{v \in W^{1}; v \nVdash_{M^1} \bigwedge \Gamma_{0}$ or $v \Vdash_{M^1} \bigvee \Gamma_{1}\} \cup \{v \in W^{2}; v \nVdash_{M^2} \bigwedge \Gamma_{0}$ or $v \Vdash_{M^2} \bigvee \Gamma_{1}\}$.

For every $x \in \bigcup \mathcal{N}^{1}_{u_1}, x \nVdash_{M_1} \bigwedge \Gamma_{0}$ - so $\bigcup \mathcal{N}^{1}_{u_1} \subseteq K$. Hence $u_1 \Vdash_{M_1} \Theta$. But at the same time $y \Vdash_{M^2} \bigwedge \Gamma_{0}$ and $y \nVdash_{M^2} \bigvee \Gamma_{1}$. Thus $\bigcup \mathcal{N}^{2}_{u_2} \nsubseteq K$, i.e. $u_2 \nVdash_{M^2} \Theta $. This is plain contradiction because $deg(\Theta) \leq i+1$.

Now consider the third possibility: that $\Gamma_{0} = \emptyset$. Then $y \nVdash_{M^2} \Theta$. Therefore $u_2 \nVdash_{M^2} \Delta \Theta$. But for every $x \in \bigcup \mathcal{N}^{1}_{w_1}$, $x \Vdash_{M^1} \Theta$. Hence $u_1 \Vdash_{M^1} \Delta \Theta$. Contradiction - because $deg(\Delta \Theta) \leq i+1$.

Those considerations refer to clause 4 but solution for clause 6 is of course very similar; the only thing is to remember about proper notation.

\end{proof}

\section{Various frame conditions}

Fischer Servi introduced (see [9]) two important conditions which connect preorder $\leq$ with modal relation $R$ in bi-relational frame for intutionistic modal logic:

\begin{enumerate}

\item \textbf{F1}. If $w \leq u$ and $w R v$ then there exists $z \in W$ such that $u R z$ and $v \leq z$

\item \textbf{F2}. If $w R v$ and $v \leq u$ then there exists $z \in W$ such that $w \leq z$ and $z R u$

\end{enumerate}

Our \textbf{nIML1}-frames do not satisfy those restrictions. Consider the first case. Take the following frame: $W = \{w, v, u\}$ where $\bigcap \mathcal{N}_{w} = \{w, u\}, \bigcup \mathcal{N}_{w} = \{w, u, v\}$ and for $u, v$ their maximal (thus also minimal) neighborhoods are just their singletons. We see that $w \leq u$ and $w R v$. But for each $z \in W$ we have $\lnot (w \leq z)$ or $\lnot (u R z)$. Let us check it. Take $z = w$. Of course $w \notin \bigcap \mathcal{N}_{v}$, thus $\lnot (v \leq w)$. The same for $z = u$. If $z = v$, then $v \notin \bigcup \mathcal{N}_{u}$, i.e. $\lnot u R v$. So \textbf{F1} does not hold in our semantics.

\begin{figure}

\centering
\begin{minipage}{.4\textwidth}
  \centering

\includegraphics[width=.7\linewidth]{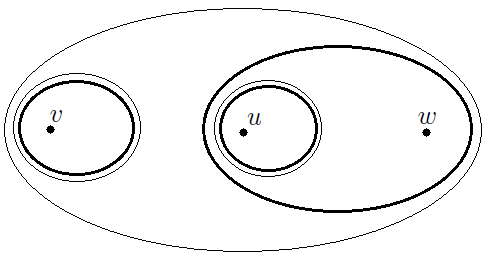}
\caption{Counter-frame for \textbf{F1}}
\label{fig:obrazek {f1axiom}}
\end{minipage}%
\begin{minipage}{.4\textwidth}
  \centering
\includegraphics[width=.7\linewidth]{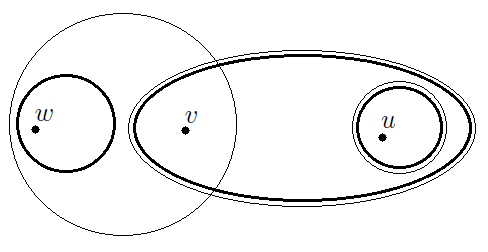}
\caption{Counter-frame for \textbf{F2}}
\label{fig:obrazek {f2axiom}}

\end{minipage}
\end{figure}

Now let us think about such frame: $W = \{w, u, v\}, \bigcap \mathcal{N}_{w} = \{w\}, \bigcup \mathcal{N}_{w} = \{w, v\}, \bigcap \mathcal{N}_{v} = \bigcup \mathcal{N}_{v} = \{v, u\}$ and for $u$ its maximal (and minimal) neighborhood is just $\{u\}$. We can say that $w R v$ because $v \in \bigcup \mathcal{N}_{w}$. Moreover, $v \leq u$ because $u \in \bigcap \mathcal{N}_{v}$. But for each $z \in W$ we have $\lnot (w \leq z)$ or $\lnot (z R u)$. Let us check it. Take $z = w$, we see that $u \notin \bigcup \mathcal{N}_{w}$, thus $\lnot (w R u)$. If $z = v$, then $v \notin \bigcap \mathcal{N}_{w}$. The same holds for $z = u$. Thus \textbf{F2} fails.

We can say that $T$-condition implies \textbf{F2}. In fact, we can always use $w$ as our expected $z$. Moreover, in our environment both conditions are equivalent. Suppose that we have \textbf{nIML1}-model with \textbf{F2} but $T$ does not hold. So we have $v, w \in W$ such that $v \in \bigcup \mathcal{N}_{w}$ but $\bigcap \mathcal{N}_{v} \nsubseteq \bigcup \mathcal{N}_{w}$. Thus there is $u \in \bigcap \mathcal{N}_{v}$ which does not belong to $\bigcup \mathcal{N}_{w}$. But then (by \textbf{F2}) there exists $z \in \bigcap \mathcal{N}_{w}$ such that $u \in \bigcup \mathcal{N}_{z}$. By $\Delta$-condition $\bigcup \mathcal{N}_{z} \subseteq \bigcup \mathcal{N}_{w}$. Now $u \in \bigcup \mathcal{N}_{w}$. This is contradiction. As we can see, $\Delta$-condition was used.

What is the meaning of \textbf{F1}? In Fischer Servi systems it assures us that monotonicity of forcing holds. In fact, it allows us to say that forcing of possibility is satisfied. Monotonicity of necessity is contained in the very definition of forcing. In our terms it would be:

$w \Vdash \Box \varphi$ $\Leftrightarrow$ $\bigcap \mathcal{N}_{w} \subseteq \{x \in W; \bigcup \mathcal{N}_{x} \Vdash \varphi\}$.

Our solution is different. We use much more simple definition of necessity but we must pay for it, assuming $\Delta$-condition in our frames. Note that in the presence of $\Delta$-condition $\Box$ and $\Delta$ become equivalent. As for possibility, we will discuss it in the next chapter. However, we can already say that \textbf{F1} will be adopted (at least in basic case).

Plotkin and Stirling in [7] discuss also two additional clauses which are less popular. We present them already in our neighborhood language.
\begin{enumerate}

\item \textbf{PS1}. If $v \in \bigcap \mathcal{N}_{w}$ and $u \in \bigcup \mathcal{N}_{v}$, then there is $z \in W$ such that $z \in \bigcup \mathcal{N}_{w}$ and $u \in \bigcap \mathcal{N}_{z}$.

This clause is always satisfied in \textbf{nIML1}-models. In fact, we can always take $u$ as $z$. Surely, $u \in \bigcup \mathcal{N}_{w}$ (by $\Delta$-condition) and $u \in \bigcap \mathcal{N}_{u}$.

\item \textbf{PS2}. If $v \in \bigcap \mathcal{N}_{w}$ and $v \in \bigcup \mathcal{N}_{u}$ then there is $z \in W$ such that $w \in \bigcup \mathcal{N}_{z}$ and $u \in \bigcap \mathcal{N}_{z}$.

This clause is not true. Check the following counter-frame: $W = \{w, v, u\}, \bigcap \mathcal{N}_{w} = \bigcup \mathcal{N}_{w} = \{w, v\}, \bigcap \mathcal{N}_{v} = \bigcup \mathcal{N}_{v} = \{v\}, \bigcap \mathcal{N}_{u} = \{u\}$ and $\bigcup \mathcal{N}_{u} = \{u, v\}$. Now: $w \in \bigcap \mathcal{N}_{w}$ and $v \in \bigcup \mathcal{N}_{u}$ but there is not any world which could be $z$. It is because $u \notin \bigcap \mathcal{N}_{w}, u \notin \bigcap \mathcal{N}_{v}$ and $w \notin \bigcup \mathcal{N}_{v}$.

\end{enumerate}

\section{Possibility operators}

\subsection{Various possibilities of possibility}

\textbf{IML1}-logic has been defined as mono-modal, i.e. in a simple language containing only one modal operator$\Delta$. For such system we proved (among other properties) semantical completeness. Now the question is: how we can define possibility operator in neighborhood framework, preserving intuitionistic aspect of our logic? Basically, we can use the following definition:

$w \Vdash \nabla \varphi$ $\Leftrightarrow$ there exists $v \in \bigcup \mathcal{N}_{w}$ such that $v \Vdash \varphi$

The problem is that such operator violates monotonicity of forcing (which is essential for intuitionism). One can easily find an example of the following situation: that $w \Vdash \nabla \varphi$, $v \in \bigcap \mathcal{N}_{w}$ and $v \nVdash \nabla \varphi$. We must remember that $\Delta$-condition guarantees monotonicity only for necessity operator. The simplest way to obtain expected property for $\nabla$ is to use \textbf{F1}-condition, described in the previous section. For convenience, we can introduce the notion of \textbf{nIML1-F1}-frame (model).

\begin{lem}
In \textbf{nIL1-F1}-model $w \Vdash \nabla \varphi$ implies that $u \Vdash \nabla \varphi$ for every $u \in \bigcap \mathcal{N}_{w}$.
\end{lem}

\begin{proof}
At first, let us rewrite \textbf{F1} in terms of neighborhood structures: if $u \in \bigcap \mathcal{N}_{w}$ and $v \in \bigcup \mathcal{N}_{w}$, then there exists $z \in W$ such that $z \in \bigcup \mathcal{N}_{u}$ and $z \in \bigcap \mathcal{N}_{v}$.

Assume now that $w \Vdash \nabla \varphi$. Thus there exists certain $v \in \bigcup \mathcal{N}_{w}$ such that $v \Vdash \varphi$. Suppose that $u \in \bigcap \mathcal{N}_{w}$. Our frame satisfies \textbf{F1}, so we have $z$ defined as earlier. Because $z \in \bigcap \mathcal{N}_{v}$, it satisfies $\varphi$ - and because it belongs to $\bigcup \mathcal{N}_{u}$, we can say that $u \Vdash \nabla \varphi$.

\end{proof}

We can list few axioms which are satisfied in our structures:

\begin{enumerate}

\item $\Delta(\varphi \rightarrow \psi) \rightarrow (\nabla \varphi \rightarrow \nabla \psi)$

\item $\lnot \nabla \bot$

\item $\nabla(\varphi \lor \psi) \rightarrow (\nabla \varphi \lor \nabla \psi)$ (\textit{distributivity over disjunction})

\item $(\nabla \varphi \rightarrow \Delta \psi) \rightarrow \Delta(\varphi \rightarrow \psi)$

\item $(\Delta \varphi \rightarrow \varphi) \land (\varphi \rightarrow \nabla \psi)$ (\textit{bi-modal version of} \textbf{T})

\end{enumerate}

Bozic and Dosen showed in [1] how to establish completeness result for certain extension of \textbf{IML1} (with possibility operator defined as our $\nabla$) - but with respect to the very specific class of bi-relational frames. In fact, they must satisfy the following conditions: \textit{i)} $\leq R = R \leq = R$ and \textit{ii)} $\leq^{-1} R \subseteq R \leq^{-1}$.

Another interesting clause has been proposed by Wijesekera in [11]. In our language it would be (let us use provisional symbol $\Diamond$):

$w \Vdash \Diamond \varphi$ $\Leftrightarrow$ $\bigcap \mathcal{N}_{w} \subseteq \{x \in W;$ there exists $y$ such that $y \in \bigcup \mathcal{N}_{x}$ and $y \Vdash \varphi\}$.

This approach is important at least for two reasons:

\begin{enumerate}

\item We can easily prove that forcing of $\Diamond \varphi$ is monotone with respect to minimal neighborhoods even without \textbf{F1}. Suppose that $w \Vdash \Diamond \varphi$ and take an arbitrary $v \in \bigcap \mathcal{N}_{w}$. As we know, $\bigcap \mathcal{N}_{v} \subseteq \bigcap \mathcal{N}_{w}$ which clearly leads us to expected conclusion.

\item $\Diamond$ (contrary to $\nabla$) is not distributive with respect to disjunction. Consider the following \textbf{nIML1-F1}-counterexample: $W = \{w, u, v\}, \bigcap \mathcal{N}_{w} = \bigcup \mathcal{N}_{w} = W, \bigcap \mathcal{N}_{u} = \bigcup \mathcal{N}_{u} = \{u\}$ and $\bigcap \mathcal{N}_{v} = \bigcup \mathcal{N}_{v} = \{v\}$. Assume that $V(\varphi) = \{u\}$ and $V(\psi) = \{v\}$. Now both $u$ and $v$ satisfy $\varphi \lor \psi$. Thus, $w \Vdash \Diamond(\varphi \lor \psi)$. But $w \nVdash \Diamond \varphi$ (because of $v$) and $w \nVdash \Diamond \psi$ (because of $u$).
\end{enumerate}

Moreover, the original Wijesekera's system is interesting because if we add the Law of Excluded Middle to it's axioms, we will not get classical modal logic with standard duality between necessity and possibility (see Simpson [9]).

There is also third interesting "possibility of possibility". Simpson in [9] assigns it to the unpublished paper of Plotkin and Stirling. At first, this definition looks strange (we present it in our neighborhood language, also we use new provisional symbol $\heartsuit$):

$w \Vdash \heartsuit \varphi$ $\Leftrightarrow$ there exist $u, v$ such that $w \in \bigcap \mathcal{N}_{u}$, $v \in \bigcup \mathcal{N}_{u}$ and $v \Vdash \varphi$

What is surprising, is the fact that we make "step back" from $w$ to our expected $u$. Thus, we say that $w$ is seen (intuitionistically) from $u$. Note, however, that such approach guarantees us monotonicity of forcing (because of $\rightarrow$-condition) - so we do not need \textbf{F1} and we can stay with our basic \textbf{nIML1}-structures.

In fact, the real problem with $\Diamond$ and $\heartsuit$ (defined as above) is rather philosophical than mathematical. They are quite too far from our simple intuition of possibility which says (hopefully, the reader will agree) that satisfaction of a given formula $\varphi$ at least in one point of the range assigned to $w$ (e.g. in $\bigcup \mathcal{N}_{w})$ suffices to say that $\varphi$ is possible in $w$. Moreover, in case of $\heartsuit$ we do not expect modal reachability of $v$ from $w$. Thus, we accept strange situation in which $w$ does not see $v$ (by no means) but "believes" $u$ that $u$ sees acceptance of $\varphi$ somewhere in \textit{its own} maximal neighborhood (i.e. in modal sense).

\subsection{U-condition}

Having $\nabla$ we can investigate an interesting restriction, named $U$-condition: $\bigcup \mathcal{N}_{w} \cap \bigcup \mathcal{N}_{v} \neq \emptyset$ $\Rightarrow$ $\bigcap \mathcal{N}_{w} \cap \bigcap \mathcal{N}_{v} \neq \emptyset$.

\begin{lem}

In \textbf{nIML1-F1}-model $U$-condition is characterized by $\lnot (\varphi \land \nabla \lnot \varphi)$.

\end{lem}

\begin{proof}
Suppose that we do not have $U$-condition. Then take $W = \{w, u\}, \bigcap \mathcal{N}_{w} = \{w\}, \bigcup \mathcal{N}_{w} = \{w, u\}, \bigcap \mathcal{N}_{u} = \bigcup \mathcal{N}_{u} = \{u\}, V(\varphi) = \{w\}$. Now $w \nVdash \lnot(\varphi \land \nabla \lnot \varphi)$. It is because we have an element of minimal $w$-neighborhood (in fact, $w$ itself) which forces $\varphi \land \nabla \lnot \varphi$. The right part of this conjunction is forced because there is $u \in \bigcup \mathcal{N}_{w}$ such that $u \Vdash \lnot \varphi$.

Now assume $U$-condition. Suppose that there is $w \in W$ such that $w \nVdash \lnot(\varphi \land \nabla \lnot \varphi)$. Then there is $v \in \bigcap \mathcal{N}_{w}$ for which conjunction in brackets holds. Then there exists $u \in \bigcup \mathcal{N}_{v}$ such that $u \Vdash \lnot \varphi$. Clearly, $\bigcup \mathcal{N}_{v} \cap \bigcup \mathcal{N}_{u} \neq \emptyset$. Thus, $\bigcap \mathcal{N}_{v} \cap \bigcap \mathcal{N}_{u} \neq \emptyset$. But $v \Vdash \varphi$ so $\bigcap \mathcal{N}_{v} \subseteq V(\varphi)$ and then $\bigcap \mathcal{N}_{v} \cap \bigcap \mathcal{N}_{u} \subseteq V(\varphi)$. But if $u \Vdash \lnot \varphi$ then by the same reasoning $\bigcap \mathcal{N}_{v} \cap \bigcap \mathcal{N}_{u} \subseteq -V(\varphi)$. Contradiction.

\end{proof}

\subsection{Question of duality}

A natural question arises: what about duality between $\Delta$ and $\nabla$? We can prove some general results:

\begin{lem}
$\Delta$ is not definable in terms of other connectives.
\end{lem}

\begin{proof}

At first, consider two \textbf{nIML1-F1}-models which are shown below:

\begin{figure}[ht]
\centering
\includegraphics[height=4cm]{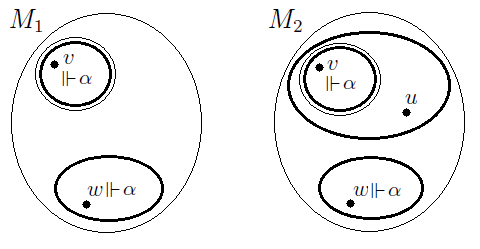}
\caption{$\Delta$ is not definable by $\nabla$ and other operators}
\label{fig:obrazek {deltanabla_1}}
\end{figure}

$M_1 = \langle W, \mathcal{N}^{1}, V^{1} \rangle$, where $W = \{w, v\}, \bigcap \mathcal{N}^{1}_{w} = \{w\}, \bigcup \mathcal{N}^{1}_{w} = \{w, v\}, \bigcap \mathcal{N}^{1}_{v} = \bigcup \mathcal{N}^{1}_{v} = \{v\}$ and $V^{1}(\alpha) = \{w, v\}$

$M_2 = \langle W \cup \{u\}, \mathcal{N}^{2}, V^{2} \rangle$, where $\bigcap \mathcal{N}^{2}_{w} = \{w\}, \bigcup \mathcal{N}^{2}_{w} = \{w, v, u\}, \bigcap \mathcal{N}^{2}_{v} = \bigcup \mathcal{N}^{2}_{v} = \{v\}, \bigcap \mathcal{N}^{2}_{u} = \bigcup \mathcal{N}^{2}_{u} = \{u, v\}$ and $V^{2}_{|W} = V^{1}$.

We can prove that for each $\varphi$ which does not contain $\Delta$ the following equivalence holds: $w \Vdash_{M_1} \varphi$ $\Leftrightarrow$ $w \Vdash_{M_2} \varphi$. The proof goes by induction on the complexity of formula.

For propositional letters we have: $w \Vdash_{M_1} q$ $\Leftrightarrow$ $w \in V^{1}(q)$ $\Leftrightarrow$ $w \in V^{2}(q)$ $\Leftrightarrow$ $w \Vdash_{M_2} q$. Conjunction and disjunction are simple.

For implication we can write: $w \Vdash_{M_1} \gamma \rightarrow \delta$ $\Leftrightarrow$ $\bigcap \mathcal{N}^{1}_{w} \subseteq \{x \in W; x \nVdash_{M_1} \gamma $ or $x \Vdash_{M_1} \delta \}$ $\Leftrightarrow$ $\{w\} \subseteq \{x \in W; x \nVdash_{M_1} \gamma $ or $x \Vdash_{M_1} \delta \}$ $\Leftrightarrow$ $\bigcap \mathcal{N}^{2}_{w} \subseteq \{x \in W \cup \{u\}; x \nVdash_{M_1} \gamma $ or $x \Vdash_{M_1} \delta \}$ $\Leftrightarrow$ $w \Vdash_{M_2} \gamma \rightarrow \delta$.

Now suppose that $\varphi := \nabla \gamma$ and $w \Vdash_{M_1} \nabla \gamma$. This means that there exists $x \in \bigcup \mathcal{N}^{1}_{w}$ such that $x \Vdash_{M_1} \gamma$. But clearly $x \in \bigcup \mathcal{N}^{2}_{w}$ (because $W \subseteq W \cup \{u\})$. By induction hypothesis $x \Vdash_{M_2} \gamma$ and hence $w \Vdash_{M_2} \nabla \gamma$. On the other side, assume that $w \Vdash_{M_2} \nabla \gamma$. Now there is $x \in \bigcup \mathcal{N}^{2}_{w}$ such that $x \Vdash_{M_2} \gamma$. There are three possibilities. If $x = w$ or $x = v$ then $x \in \bigcup \mathcal{N}^{1}_{w}$. By induction hypothesis, $x \Vdash_{M_1} \gamma$ and hence $w \Vdash_{M_2} \nabla \gamma$. If $x = u$, then $v \Vdash_{M_2} \gamma$ (because $v$ belongs to the minimal $u$-neighborhood). At this moment we can repeat earlier reasoning.

The crucial thing is that $w \Vdash_{M_1} \Delta \alpha$ but $w \nVdash_{M_2} \Delta \alpha$ (because $u$ does not accept $\alpha$ and $u \in \bigcup \mathcal{N}^{2}_{w}$).

\end{proof}

Note that the last counter-model shows us also that $\rightsquigarrow$ is not definable by standard connectives: $w \Vdash_{M_1} \top \rightsquigarrow \alpha$ and $w \nVdash_{M_2} \top \rightsquigarrow \alpha$ because of $u$. By $\top$ we understand arbitrary tautology of \textbf{L1-F1}-logic.

\begin{lem}
$\nabla$ is not definable in terms of other connectives.
\end{lem}

\begin{proof}
Consider two \textbf{nIML1-F1}-models which are shown below:

\begin{figure}[ht]
\centering
\includegraphics[height=4cm]{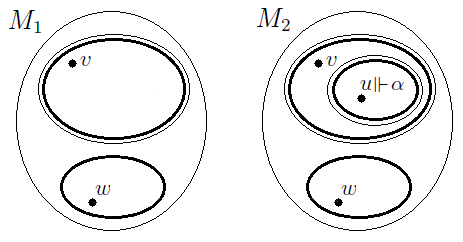}
\caption{$\nabla$ is not definable by $\Delta$ and other operators}
\label{fig:obrazek {deltanabla_1}}
\end{figure}

$M_1 = \langle W, \mathcal{N}^{1}, V^{1} \rangle$, where $W = \{w, v\}, \bigcap \mathcal{N}^{1}_{w} = \{w\}, \bigcup \mathcal{N}^{1}_{w} = \{w, v\}, \bigcap \mathcal{N}^{1}_{v} = \bigcup \mathcal{N}^{1}_{v} = \{v\}$ and $V^{1}(\alpha) = \emptyset$

$M_2 = \langle W \cup \{u\}, \mathcal{N}^{2}, V^{2} \rangle$, where $\bigcap \mathcal{N}^{2}_{w} = \{w\}, \bigcup \mathcal{N}^{2}_{w} = \{w, v, u\}, \bigcap \mathcal{N}^{2}_{v} = \bigcup \mathcal{N}^{2}_{v} = \{v, u\}, \bigcap \mathcal{N}^{2}_{u} = \bigcup \mathcal{N}^{2}_{u} = \{u\}$ and $V^{2} (\alpha) = \{u\}$.

We can prove that for each $\varphi$ which does not contain $\nabla$ the following equivalence holds: $w \Vdash_{M_1} \varphi$ $\Leftrightarrow$ $w \Vdash_{M_2} \varphi$. The proof goes by induction on the complexity of formula. Propositional case is easy, the same can be said about $\land$, $\lor$ and $\rightarrow$. What about $\Delta$? Check the following sequence of meta-equivalences:

$w \Vdash_{M_1} \Delta \varphi$ $\Leftrightarrow$ $\bigcup \mathcal{N}^{1}_{w} \Vdash_{M_1} \varphi$ $\Leftrightarrow$ $\{w, v\} \Vdash_{M_1} \varphi$ $\Leftrightarrow$ $\{w, v\} \Vdash_{M_2} \varphi$ $\Leftrightarrow$ $\{w, v, u\} \Vdash_{M_2} \varphi$ $\Leftrightarrow$ $\bigcup \mathcal{N}^{2}_{w} \Vdash_{M_2} \varphi$ $\Leftrightarrow$ $w \Vdash_{M_2} \Delta \varphi$.

We used induction hypothesis and the fact that $u$ is in minimal $v$-neighborhood - thus acceptation of certain formula in $v$ implies its acceptation in $u$. Now think about $\nabla$-case. Obviously, $w \nVdash_{M_1} \nabla \alpha$ but $w \Vdash_{M_2} \nabla \alpha$.

\end{proof}

One detail should be noted: it turns out that if we accept $T$-condition, then $\lnot \Delta \lnot \varphi \rightarrow \nabla \varphi$ becomes valid. Check it. Suppose that in certain particular \textbf{nIML2-F1}-model we have $w$ such that $w \Vdash \lnot \Delta \lnot \varphi$ and $w \nVdash \nabla \varphi$. Then from the left side: for each $v \in \bigcap \mathcal{N}_{w}, v \nVdash \Delta \lnot \varphi$. This means that there exists $u \in \bigcup \mathcal{N}_{v}$ such that $u \nVdash \lnot \varphi$. So there must be $x \in \bigcap \mathcal{N}_{u}$ such that $x \Vdash \varphi$. Now take a look on the right side for a moment. It says that $\bigcup \mathcal{N}_{w} \nVdash \varphi$. Think about our $u$. We said that $v \in \bigcap \mathcal{N}_{w}$. Thus, $\bigcup \mathcal{N}_{v} \subseteq \bigcup \mathcal{N}_{w}$ (by $\Delta$-condition). Now we use $T$-condition: $u \in \bigcup \mathcal{N}_{v}$ so $\bigcap \mathcal{N}_{u} \subseteq \bigcup \mathcal{N}_{v} \subseteq \bigcup \mathcal{N}_{w}$. Hence, $x \nVdash \varphi$. Contradiction.

Attention: in the previous chapter we proved that in our setting (i.e. in the presence of $\Delta$-condition) $T$-condition and \textbf{F2} are equivalent. Thus, \textbf{nIML2-F1}-frame can be considered as \textbf{nIML1-F1-F2}-frame. Hence, \textbf{nIML2-F1}-frames are just subclass of Fischer-Servi structures - but satisfying $\Delta$-condition (which is quite strong).

Simpson says in [9] that \textbf{F2} condition means exactly that formulas such as $\lnot \Diamond \varphi \rightarrow \Delta \lnot \varphi$ hold. In our language it would be $\lnot \nabla \varphi \rightarrow \Delta \lnot \varphi$. Note that this formula is equivalent to $\sim \varphi \rightarrow \Delta \lnot \varphi$ (which was mentioned in chapter 6).

\section{Some topological issues}

One can easily see that our neighborhood structure does not form a topology, i.e. we cannot treat $w$-neighborhoods as open sets. It is because we do not have superset axiom which is necessary for any proper definition of the family of open sets. However, we have some clues (or maybe intuition) that there can be certain analogy between $\bigcup \mathcal{N}_{w}$ and the idea of topological space. It seems that we can consider our universe as a sum of such maximal neighborhoods. Those 'maximal bubbles' can intersect or form unions.

This leads us to the following definition:

\begin{df}
We say that set $X \subseteq W$ is $w$-open iff $X \subseteq \bigcup \mathcal{N}_{w}$ and for every $v \in X$ we have $\bigcap \mathcal{N}_{v} \subseteq X$. We denote $\mathcal{U}_{w} = \{X \subseteq W; X $ are $w$-open $\}$.
\end{df}

But almost immediately we have first problem. Think about $X = \bigcup \mathcal{N}_{w}$. We want it to be open. It is because we treat this set as our 'relative universe', as the whole space. It is obvious that $X \subseteq \bigcup \mathcal{N}_{w}$ but we do not have any guarantee (in general case) that $\bigcap \mathcal{N}_{v} \subseteq X$ for every $v \in X$. We must use 'brute force' and assume that we consider only \textit{topologically suitable} models. In fact, such models must be \textbf{nIML2}-models, i.e. satisfy $T$-condition. Recall this restriction: if $v \in \bigcup \mathcal{N}_{w}$ then $\bigcap \mathcal{N}_{v} \subseteq \bigcup \mathcal{N}_{w}$. It is obvious that it is useful for our needs. Also, it is possible to formulate it in different but equivalent way - just as it was with $\rightarrow$- and $\Delta$-conditions.

\begin{lem}
This condition: $\bigcup \mathcal{N}_{w} \subseteq X \Rightarrow \bigcup \mathcal{N}_{w} \subseteq \{v \in W; \bigcap \mathcal{N}_{w} \subseteq X\}$ implies \textit{T}-condition.
\end{lem}

\begin{proof}
Suppose that $u \in \bigcup \mathcal{N}_{w}$. Of course $\bigcup \mathcal{N}_{w} \subseteq \bigcup \mathcal{N}_{w}$ and then (by our assumption) $\bigcup \mathcal{N}_{w} \subseteq \{z \in W;$ $\bigcap \mathcal{N}_{z} \subseteq \bigcup \mathcal{N}_{w}\}$. So if $u \in \bigcup \mathcal{N}_{w}$ then $\bigcap \mathcal{N}_{u} \subseteq \bigcup \mathcal{N}_{w}$.
\end{proof}

\begin{lem}
\textit{T}-condition implies the following: $\bigcup \mathcal{N}_{w} \subseteq X \Rightarrow \bigcup \mathcal{N}_{w} \subseteq \{v \in W; \bigcap \mathcal{N}_{w} \subseteq X\}$.
\end{lem}

\begin{proof}
Take $r \in W$ and assume that $\bigcup \mathcal{N}_{r} \subseteq X \subseteq W$. Now consider $Y = \{z \in W;$ $\bigcap \mathcal{N}_{z} \subseteq X\}$.

Suppose that $\bigcup \mathcal{N}_{r} \nsubseteq Y$, so there exists certain $s \in \bigcup \mathcal{N}_{r}$ that $s \notin Y$. But this gives us that $\bigcap \mathcal{N}_{s} \nsubseteq X$, so there is $p \in \bigcap \mathcal{N}_{s}$ such that $p \notin X$. But if $s \in \bigcup \mathcal{N}_{r}$ then $\bigcap \mathcal{N}_{s} \subseteq \bigcup \mathcal{N}_{r} \subseteq X$. Then we have that $p \in X$ - and this is contradiction.
\end{proof}

Now let us check our definition of topology.

\begin{tw}
Assume that we have suitable neighborhood model $\langle W, \mathcal{N}, V \rangle$. Then $\mathcal{U}_{w}$ is a topological space for every $w \in W$.
\end{tw}

\begin{proof}
Let us check standard properties of well-defined topology.
\begin{enumerate}

\item Take empty set. We can say that $\emptyset \in \mathcal{U}_{w}$ because $\emptyset \subseteq \bigcup \mathcal{N}_{w}$ and the second condition is also true because there are no any $v$ in $\emptyset$.

\item Consider the whole 'relative universe', i.e. $\bigcup \mathcal{N}_{w}$. Clearly this set is contained in itself and from topological suitability we have that for every $v \in \bigcup \mathcal{N}_{w}$ the second condition holds: $\bigcap \mathcal{N}_{v} \subseteq \bigcup \mathcal{N}_{w}$.

\item Consider $\mathscr{X} \subseteq \mathcal{U}_{w}$. We show that $\bigcap \mathscr{X} \in \mathcal{U}_{w}$. The first condition is simple: every element of $\mathscr{X}$ belongs to $\mathcal{U}_{w}$ so it is contained in $\bigcup \mathcal{N}_{w}$. The same holds of course for intersection of all such elements.

    Now let $v \in \bigcap \mathscr{X}$. By the definition we have that $\bigcap \mathcal{N}_{v} \subseteq X$ for every $X \in \mathscr{X}$. Then $\bigcap \mathcal{N}_{v} \subseteq \bigcap \mathscr{X}$.

\item In the last case we deal with arbitrary unions. Suppose that $\mathscr{X} \subseteq \mathcal{U}_{w}$ and consider $\bigcup \mathscr{X}$. Surely this union is contained in $\bigcup \mathcal{N}_{w}$. Now let us take an arbitrary $v \in \bigcup \mathscr{X}$. We know that $\bigcap \mathcal{N}_{v} \subseteq X$ for some $X \in \mathscr{X}$ (in fact, it holds for every $X$ which contains $v$). Then clearly $\bigcap \mathcal{N}_{v} \subseteq \bigcup \mathscr{X}$.

\end{enumerate}

\end{proof}

Note that for every $w \in W$ $\mathcal{U}_{w}$ is an Alexandroff space, i.e. intersection of any family of open sets is open. In general, those topological spaces can be very different - and we cannot expect even $T_0$-property. For example, take $W = \{w, v\}$ where $w$ and $v$ have only one (maximal and thus also minimal) neighborhood, the same for both elements, and it is just $W$. Now $W$ is the only $w$-open set (or $v$-open set).

One could say that it would be also very intuitive to define $w$-open set as such: $X \subseteq W$ is $w$-open $\Leftrightarrow$ $X \subseteq \bigcup \mathcal{N}_{w}$ and ($X = \emptyset$ or $\bigcap \mathcal{N}_{w} \subseteq X$). Such definition does not require our suitability condition - but it states that non-empty $w$-open sets are just $w$-neighborhoods (with their sums and intersections). Such topology would be completely blind for the existence of other neighborhoods.

\section{Certain classical cases}

\subsection{Basic definitions}

In this section our approach to structures introduced earlier can be described as classical. It means that we use axioms of classical propositional calculus (\textbf{CPC}) as our syntactic foundation - and we do not expect monotonicity of forcing in model. On the other hand, we introduce two modal operators: the first is just like $\Delta$ (e.g. it refers to satisfaction of formula in the whole maximal neighborhood), the second one describes forcing in minimal neighborhood. Thus, it is like $\Box$. In fact, we even use this symbol. As we remember, $w \Vdash \Box \varphi$ means - in classical neighborhood semantics - that there exists certain $X \in \mathcal{N}_{w}$ such that $X \Vdash \varphi$. But if $\bigcap \mathcal{N}_{w} \in \mathcal{N}_{w}$, then it is equivalent to say that $\bigcap \mathcal{N}_{w} \Vdash \varphi$.

\begin{df}
Our basic structure is a \textbf{nCL1}-frame for \textbf{CL1} logic. Such frame is defined as an ordered pair $\langle W, \mathcal{N} \rangle$ where:

\begin{enumerate}
\item $W$ is a non-empty set (of worlds, states or points)

\item $\mathcal{N}$ is a function from $W$ into $P(P(W))$ such that:

\begin{enumerate}

\item $w \in \bigcap \mathcal{N}_{w}$

\item $\bigcap \mathcal{N}_{w} \in \mathcal{N}_{w}$

\item $X \subseteq \bigcup \mathcal{N}_{w}$ and $\bigcap \mathcal{N}_{w} \subseteq X$ $\Rightarrow$ $X \in \mathcal{N}_{w}$ (\textit{relativized superset axiom})

\end{enumerate}

\end{enumerate}

\end{df}

Note that we discarded $\rightarrow$-condition, i.e. nesting of minimal neighborhoods. Having frame, we define model:

\begin{df}
Neighborhood \textbf{nCL1}-model is a triple $\langle W, \mathcal{N}, V \rangle$, where $\langle W, \mathcal{N} \rangle$ is an \textbf{nCL1}-frame and $V$ is a function from $PV$ into $P(W)$.
\end{df}

\begin{df}
Forcing of formulas in a world $w \in W$ is defined inductively:

\begin{enumerate}

\item $w \Vdash q$ $\Leftrightarrow$ $w \in V(q)$ for any $q \in PV$

\item $w \Vdash \varphi \rightarrow \psi$ $\Leftrightarrow$ $w \nVdash \varphi$ or $w \Vdash \psi$.

\item $w \Vdash \lnot \varphi$ $\Leftrightarrow$ $w \nVdash \varphi$

\item $w \Vdash \Box \varphi$ $\Leftrightarrow$ $\bigcap \mathcal{N}_{w} \subseteq \{v \in W; v \Vdash \varphi\}$

\item $w \Vdash \Delta \varphi$ $\Leftrightarrow$ $\bigcup \mathcal{N}_{w} \subseteq \{v \in W; v \Vdash \varphi\}$

\item $w \nVdash \bot$

\end{enumerate}
\end{df}

Note that such model has certain non-mathematical interpretation. We can treat each world as a person which spreads certain opinion or information among his neighbors, having minimal and maximal expected target group (audience). As for $\Delta$, we can say that it behaves like relativized universal modality (see Pacuit [5]).

Which (modal) axioms are true in this class? Here we have those which are essential for the proof of completeness theorem:

\begin{enumerate}

\item Axioms involving only $\Box$ as modal operator:

\begin{enumerate}

\item \textbf{K}$\Box$: $\Box(\varphi \rightarrow \psi) \rightarrow (\Box \varphi \rightarrow \psi)$

\item \textbf{T}$\Box$: $\Box \varphi \rightarrow \varphi$ (note that $w \in \bigcap \mathcal{N}_{w}$)

\item \textbf{RN} (\textit{necessity rule}): $\varphi$ $\Rightarrow$ $\Box \varphi$

\end{enumerate}

\item Axioms involving only $\Delta$ as modal operator: $\Delta$-analogues of \textbf{K}, \textbf{T} and \textbf{RN}

\item Axioms involving both $\Box$ and $\Delta$: $\Delta \varphi \rightarrow \Box \varphi$ ($\Delta$\textbf{N})

\end{enumerate}

\subsection{Completeness and canonical model}

It is probably obvious that semantic version of deduction theorem still holds, just like theorem about expansion of each theory to relatively maximal prime-theory. Thus we can define canonical model:

\begin{df}

We define canonical \textbf{nCL1}-model as a triple $\langle W, \mathcal{N}, V \rangle$ where $W$, $\mathcal{N}$ and $V$ are defined as presented below.

\begin{enumerate}

\item $W$ is the set of all \textbf{CL1} prime theories.

\item Neighborhood function for canonical model is a mapping $\mathcal{N}: W \rightarrow P(P(W))$ such that

$X \in \mathcal{N}_{w}$ $\Leftrightarrow$ $\bigcap \mathcal{N}_{w} \subseteq X \subseteq \bigcup \mathcal{N}_{w}$, where:

\begin{enumerate}

\item $\bigcap \mathcal{N}_{w} = \{v \in W: \Box \varphi \in w \Rightarrow \varphi \in v\}$.

\item $\bigcup \mathcal{N}_{w} = \{v \in W: \Delta \varphi \in w \Rightarrow \varphi \in v\}$.
\end{enumerate}

\item $V$ is a valuation defined in standard manner, i.e. as a function $V: PV \rightarrow P(W)$ such that $w \in V(q)$ $\Leftrightarrow$ $q \in w$.

\end{enumerate}

\end{df}

Of course we must prove that such structure actually is \textbf{nCL1}-model.

\begin{lem}
Canonical neighborhood model is defined in a proper way.
\end{lem}

\begin{proof}

We must check conditions from the definition of frame:

\begin{enumerate}

\item $w \in \bigcap \mathcal{N}_{w}$. This restriction holds because $w$ is theory, thus: if $\Box \varphi \in w$ then (by \textbf{T}$\Box$ and \textit{modus ponens}) $\varphi \in w$.

\item $\bigcap \mathcal{N}_{w} \in \mathcal{N}_{w}$. We must show that $\bigcap \mathcal{N}_{w} \subseteq \bigcup \mathcal{N}_{w}$. Suppose that $v \in \bigcap \mathcal{N}_{w}$ and that $\Delta \varphi \in w$. Because of $\Delta$\textbf{N} axiom and \textit{modus ponens} we have that $\Box \varphi \in w$. Thus $\varphi \in v$.

\item $\bigcap \mathcal{N}_{w} \subseteq X \subseteq \bigcup \mathcal{N}_{w}$ $\Rightarrow$ $X \in \mathcal{N}_{w}$. Relative superset axiom holds by the very definition of canonical neighborhood function.
\end{enumerate}
\end{proof}

Now we can prove the following theorem:

\begin{tw}
If $\langle W, \mathcal{N}, V \rangle$ is a canonical model (defined as above) then for each $w \in W$ we have the following equivalence: $w \Vdash \varphi$ $\Leftrightarrow$ $\varphi \in w$.

\end{tw}

\begin{proof} (sketch)
The proof goes by the induction over the complexity of formulas. We have two really important cases: $\Delta$ and $\Box$. But each of them can be provided almost exactly like proof for $\Delta$ in case of \textbf{nIML1} models and \textbf{IML1} logic. In particular, it means that for $\Rightarrow$-cases we must use $\Delta^{-1}w$ (resp. $\Box^{-1}w$) set and extension lemma. The other side is obvious: if $w \nVdash \Box \varphi$ (resp. $w \nVdash \Delta \varphi$) then there is $v \in \bigcap \mathcal{N}_{w}$ (resp. $v \in \bigcup \mathcal{N}_{w}$) such that $v \nVdash \varphi$. By induction, $\varphi \notin v$. But now, if we assume that $\Delta \varphi$ (resp. $\Box \varphi$) belongs to $w$ then we obtain clear contradiction with the very definition of maximal (minimal) neighborhood in canonical model.

\end{proof}

\subsection{Additional conditions}

Now let us assume that we accept certain restriction on the frame - named earlier $\Delta$-condition: $v \in \bigcap \mathcal{N}_{w} \Rightarrow \bigcup \mathcal{N}_{v} \subseteq \bigcup \mathcal{N}_{w}$. We say that \textbf{nCL1}-models with this property are \textbf{nCL2}-models. Corresponding logic is of course \textbf{CL2}.

In intuitionistic case, $\Delta$-condition was immediately satisfied in canonical model because minimal neighborhood of theory was just collection of its super-theories. In classical case with $\Box$ modality (but without preservation of forcing) we must add new axiom: $\Delta \varphi \rightarrow \Box \Delta \varphi$. It is not satisfied in standard \textbf{nCL1}-model (just recall counter-model from the section about importance of $\Delta$-condition; here reasoning can be similar). Let us check it in case of canonical \textbf{nCL2}-model:

\begin{lem}
Canonical \textbf{nCL2} model satisfies $\Delta$-condition.
\end{lem}

\begin{proof}
Suppose that $v \in \bigcap \mathcal{N}_{w}$ and that $\bigcup \mathcal{N}_{v} \nsubseteq \bigcup \mathcal{N}_{w}$. Hence, we have $z \in \bigcup \mathcal{N}_{v}$ such that for certain $\varphi$, $\Delta \varphi \in w$ but $\varphi \notin v$. Note, however, that if $\Delta \varphi \in w$ then $\Box \Delta \varphi \in w$. Thus $\Delta \varphi \in v$ (because $v$ is in minimal $w$-neighborhood) and from that $\varphi \in z$ (because $z$ is in maximal $v$-neighborhood). Contradiction.

\end{proof}

Now we can say that \textbf{CL2} is complete with respect to the class of all \textbf{nCL2}-frames. What about $T$-condition, introduced in our intuitionistic investigations? Consider the following axiom: $\Delta \varphi \rightarrow \Delta \Box \varphi$. It is not satisfied in \textbf{nCL1} or \textbf{nCL2} model. It becomes true when we limit ourselves to the class of \textbf{nCL3}-frames (models), i.e. those which validate $T$-condition (but not necessarily $\Delta$-condition). Let us check that the canonical model for corresponding \textbf{CL3}-logic satisfies this restriction.

\begin{lem}
Canonical \textbf{nCL3} satisfies $T$-condition, i.e. $v \in \bigcup \mathcal{N}_{w} \Rightarrow \bigcap \mathcal{N}_{v} \subseteq \bigcup \mathcal{N}_{w}$.
\end{lem}

\begin{proof}
Suppose that $v \in \bigcup \mathcal{N}_{w}$ and $z \in \bigcap \mathcal{N}_{v}$. If $\Delta \varphi \in w$, then (by the newly introduced axiom and \textit{modus ponens}) $\Delta \Box \varphi \in w$. This gives us that $\Box \varphi \in v$ and then $\varphi \in z$ (because $z$ is in minimal $v$-neighborhood).
\end{proof}

Thus, \textbf{CL3} logic is complete with respect to the class of all \textbf{nCL3}-models.

Another useful axiom is taken from well-known \textbf{S4}-logic: $\Box \varphi \rightarrow \Box \Box \varphi$. In our setting it is equivalent with $\rightarrow$-condition. It is easy to prove completeness in this case (we name such system \textbf{CL4} and corresponding model is \textbf{nCL4}).

Let us take those characterizations together:

\begin{enumerate}

\item \textbf{CL1} = \textbf{CPC}, \textbf{K}$\Box$, \textbf{T}$\Box$, \textbf{RN}$\Box$, \textbf{K}$\Delta$, \textbf{T}$\Delta$, \textbf{RN}$\Delta$, $\Delta$\textbf{N}

\item \textbf{CL2} = \textbf{CL1} $\cup$ $\{\Delta \varphi \rightarrow \Box \Delta \varphi\}$; specific frame condition (\textit{sfc}): $v \in \bigcap \mathcal{N}_{w} \Rightarrow \bigcup \mathcal{N}_{v} \subseteq \bigcup \mathcal{N}_{w}$

\item \textbf{CL3} = \textbf{CL1} $\cup$ $\{\Delta \varphi \rightarrow \Delta \Box \varphi\}$; sfc: $v \in \bigcup \mathcal{N}_{w} \Rightarrow \bigcap \mathcal{N}_{v} \subseteq \bigcup \mathcal{N}_{w}$

\item \textbf{CL4} = \textbf{CL1} $\cup$ $\{\Box \varphi \rightarrow \Box \Box \varphi\}$; sfc: $v \in \bigcap \mathcal{N}_{w}$ $\Rightarrow$ $\bigcap \mathcal{N}_{v} \subseteq \bigcap \mathcal{N}_{w}$

\item \textbf{CL5} = \textbf{CL1} $\cup$ $\{\Delta \varphi \rightarrow \Delta \Delta \varphi\}$; sfc: $v \in \bigcup \mathcal{N}_{w}$ $\Rightarrow$ $\bigcup \mathcal{N}_{v} \subseteq \bigcup \mathcal{N}_{w}$

\end{enumerate}

Those systems are complete with respect to their corresponding classes of frames. Of course we can also mix conditions and consider, for example, logic \textbf{CL2.4} with frame conditions which are exactly like in \textbf{nIML1}-model. It is important because we can establish translation between \textbf{IML1} and \textbf{CL2.4}.

\subsection{Translation between intuitionistic and classical settings}

\begin{df}
Suppose that $\varphi \in$ \textbf{IML1}. Let us define translation $*$ from \textbf{IML1} into \textbf{CL2.4} in the following way:

$\varphi^{*} :=$
\begin{enumerate}

\item $\Box q$, $q \in PV$

\item $\gamma^{*} \land \delta^{*}$, $\varphi: = \gamma \land \delta$

\item $\gamma^{*} \lor \delta^{*}$, $\varphi: = \gamma \lor \delta$

\item $\Box (\gamma^{*} \rightarrow \delta^{*})$, $\varphi: = \gamma \rightarrow \delta$

\item $\varphi$, $\varphi := \Delta \gamma$
\end{enumerate}
\end{df}

Having this function, we can prove the theorem below:

\begin{tw}

Formula $\varphi$ is a tautology of \textbf{IML1} $\Leftrightarrow$ $\varphi^{*}$ is a tautology of \textbf{CL2.4}.
\end{tw}

\begin{proof}
$\Rightarrow$

Suppose that $\varphi \in$ \textbf{IML1} but $\varphi^{*} \notin$ \textbf{CL2.4}. Thus there is particular \textbf{nCL2.4}-model $M = \{W, \mathcal{N}, V\}$ such that $w \nVdash \varphi^{*}$ for certain $w \in W$.

Now we build \textbf{nIML1}-model $\overline{M} = \{\overline{W}, \overline {\mathcal{N}}, \overline{V}\}$. We assume that each $v \in W$ has its counterpart $\overline{v} \in \overline{W}$ such that:

\begin{enumerate}

\item $\overline{v} \Vdash q$ $\Leftrightarrow$ $v \Vdash \Box q$, for each $q \in PV$

\item $\bigcap \mathcal{N}_{\overline{v}} = \{\overline{u} \in \overline{W}; u \in \bigcap \mathcal{N}_{u}\}$

\item $\bigcup \mathcal{N}_{\overline{v}} = \{\overline{u} \in \overline{W}; u \in \bigcup \mathcal{N}_{u}\}$

\end{enumerate}

We can easily check that $\overline{M}$ is a proper \textbf{nIML1}-model. Just recall that \textbf{nCL2.4}-models satisfy $\rightarrow$-condition, $\Delta$-condition and relative superset axiom. What about monotonicity of valuation in $\overline{M}$? Having necessary frame conditions, we show it only for propositional variables. Suppose that $\overline{v} \Vdash q$. Thus $v \Vdash \Box q$, hence $\bigcap \mathcal{N}_{v} \subseteq V(q)$. Assume that there is certain $\overline{r} \in \bigcap \mathcal{N}_{\overline{v}}$ such that $\overline{v} \nVdash q$. Then we have $r \in \bigcap \mathcal{N}_{v}$ for which $r \nVdash q$. Contradiction.

Now we show that for each $\gamma \in$ \textbf{IML1} and each $v \in W$ we have:

$\overline{v} \Vdash \gamma$ $\Leftrightarrow$ $v \Vdash \gamma^{*}$

This is obvious for propositional variables, by the very definition. By induction, one can easily prove this equivalence for conjunctions and disjunctions. Suppose that $\gamma := \xi \rightarrow \delta$. We can write the following series of equivalences:

$\overline{v} \Vdash \xi \rightarrow \delta$ $\Leftrightarrow$ $\bigcap \mathcal{N}_{\overline{v}} \subseteq \{\overline{u} \in \overline{W}; \overline{u} \nVdash \xi$ or $\overline{u} \Vdash \delta\}$ $\Leftrightarrow$ $\bigcap \mathcal{N}_{v} \subseteq \{u \in W; u \nVdash \xi^{*}$ or $u \Vdash \delta^{*}\}$ $\Leftrightarrow$ $v \Vdash \Box (\xi^{*} \rightarrow \delta^{*})$ $\Leftrightarrow$ $v \Vdash \gamma^{*}$.

Something similar can be written for $\gamma := \Delta \xi$: $\overline{v} \Vdash \Delta \xi$ $\Leftrightarrow$ $\bigcup \mathcal{N}_{\overline{v}} \subseteq \{\overline{u} \in \overline{W}; u \Vdash \xi\}$ $\Leftrightarrow$ $\bigcup \mathcal{N}_{v} \subseteq \{u \in W; u \Vdash \xi\}$ $\Leftrightarrow$ $v \Vdash \Delta \xi$ $\Leftrightarrow$ $v \Vdash \gamma^{*}$.

Finally, we can say that from our assumption that there is certain \textbf{nCL2.4}-model $M$ which falsifies $\varphi$ we obtained intuitionistic model which also falsifies $\varphi$ (because $\overline{w} \nVdash \varphi$). This is contradiction because we assumed also that $\varphi$ is \textbf{IML1}-tautology.

$\Leftarrow$

Now suppose that there is $\varphi$ such that $\varphi \notin$ \textbf{IML1} but $\varphi^{*} \in$ \textbf{CL2.4}. So we have certain \textbf{nIML1}-model $\overline{M} = \{\overline{W}, \overline {\mathcal{N}}, \overline{V}\}$ with $\overline{w} \in \overline{W}$ such that $\overline{w} \nVdash \varphi$. Now let us build another model, \textbf{nCML2.4} one, named $M = \{W, \mathcal{N}, V\}$ in the following way:

\begin{enumerate}

\item $v \Vdash q$ $\Leftrightarrow$ $\overline{v} \Vdash q$, for each $q \in PV$

\item $\bigcap \mathcal{N}_{v} = \{u \in W; \overline{u} \in \bigcap \mathcal{N}_{\overline{v}}\}$

\item $\bigcup \mathcal{N}_{v} = \{u \in W; \overline{u} \in \bigcup \mathcal{N}_{\overline{v}}\}$

\end{enumerate}

It is easy to show that this model is actually \textbf{nCML2.4}. Thus we can show that for each $\gamma \in$ \textbf{IML1} and each $v \in W$ we have:

$v \Vdash \gamma^{*}$ $\Leftrightarrow$ $\overline{v} \Vdash \gamma$

For propositional variables:

$v \Vdash q^{*}$ $\Leftrightarrow$ $v \Vdash \Box q$ $\Leftrightarrow$ $\bigcap \mathcal{N}_{v} \subseteq \{u \in W; u \Vdash q\}$ $\Leftrightarrow$ $\bigcap \mathcal{N}_{\overline{v}} \subseteq \{\overline{u} \in \overline{W}; \overline{u} \Vdash q\}$ $\Leftrightarrow$ $\overline{v} \Vdash q$ (recal that intuitionism preserves monotonicity of forcing w.r.t. minimal neighborhoods).

Cases of conjunction and disjunction are easy (by induction). Now suppose that $\gamma := \xi \rightarrow \delta$. Then:

$v \Vdash \gamma^{*}$ $\Leftrightarrow$ $v \Vdash \Box (\xi^{*} \rightarrow \delta^{*})$ $\Leftrightarrow$ $\bigcap \mathcal{N}_{v} \subseteq \{u \in W; u \nVdash \xi^{*}$ or $u \Vdash \delta^{*}\}$ $\Leftrightarrow$ $\bigcap \mathcal{N}_{\overline{v}} \subseteq \{\overline{u} \in \overline{W}; \overline{u} \nVdash \xi$ or $\overline{u} \Vdash \delta\}$ $\Leftrightarrow$ $\overline{v} \Vdash \xi \rightarrow \delta$ $\Leftrightarrow$ $\overline{v} \Vdash \gamma$.

For $\gamma := \Delta \xi$ the reasoning is similar, even shorter (recall that $(\Delta \xi)^{*} = \Delta \xi$).

Finally, we have classical model validating the same formulas as $\overline{M}$. Thus, if $\overline{w} \nVdash \varphi$, then $w \nVdash \varphi^{*}$, where $w \in W$. Contradiction, because we assumed that $\varphi^{*}$ is \textbf{nCL2.4}-tautology.

\end{proof}

It is obvious that if we add $\Delta \varphi \rightarrow \Delta \Delta \varphi$ to \textbf{IML1} then our function $*$ will become translation between this extended version of \textbf{IML1} and the system \textbf{CL2.4.5}.

\end{document}